\documentclass[11pt]{article}
	\usepackage
	[onehalfspacing]
	{setspace}
	\usepackage[T1]{fontenc}
	\usepackage[utf8]{inputenc}
	\usepackage[english]{babel}
	\usepackage[stretch=10]{microtype} 
	\usepackage{mathtools}
	\mathtoolsset{centercolon}
	\usepackage{amsthm}
	\usepackage{amssymb}
	\usepackage[shortlabels]{enumitem}
	\usepackage{geometry}
	\usepackage{authblk}
	\usepackage{url}
	\usepackage{xcolor}
	\usepackage[date=year,doi=false,isbn=false,url=false, maxbibnames=99,giveninits=true, backend=bibtex,sorting=nyt]{biblatex}
	\renewbibmacro{in:}{}
	\DeclareNameAlias{author}{family-given}
	\usepackage{csquotes}
	\usepackage{hyperref}
	\hypersetup{ 
		colorlinks,
		linkcolor={red!75!black},
		citecolor={blue!75!black},
		urlcolor={black}
	}
	\usepackage{cleveref}

\allowdisplaybreaks[1]
	
	\makeatletter
	\newcommand{\subalign}[1]{%
	  \vcenter{%
		\Let@ \restore@math@cr \default@tag
		\baselineskip\fontdimen10 \scriptfont\tw@
		\advance\baselineskip\fontdimen12 \scriptfont\tw@
		\lineskip\thr@@\fontdimen8 \scriptfont\thr@@
		\lineskiplimit\lineskip
		\ialign{\hfil$\m@th\scriptstyle##$&$\m@th\scriptstyle{}##$\hfil\crcr
		  #1\crcr
		}%
	  }%
	}
	\makeatother

\newcommand{\N}{\mathbb{N}}
\newcommand{\R}{\mathbb{R}}
\newcommand{\LL}{\mathcal{L}}
\DeclarePairedDelimiter{\set}{\{}{\}}
\DeclarePairedDelimiter{\abs}{\lvert}{\rvert}
\DeclarePairedDelimiter{\norm}{\lVert}{\rVert}
\newcommand{\kset}{\mathcal{L}_k}
\newcommand{\Ex}{\mathbf{E}}
\newcommand{\Var}{\mathbf{Var}}
\newcommand{\Prob}{\mathbf{P}}
\newcommand{\compset}{\mathcal{L}_{k}}
\newcommand{\oneton}{[n]}
\newcommand{\onetom}{[m]}

\newcommand{\card}{\abs}

\newcommand{\oneover}[1]{\frac{1}{#1}}
\newcommand{\indic}[1]{\mathbf{1}_{#1}}

\newcommand{\eps}{\varepsilon}
\newcommand{\anti}[1]{\mathcal{L}(#1,\eps)}
\newcommand{\scalar}[1]{\langle #1 \rangle}
\renewcommand{\phi}{\varphi}

\newtheorem{lemma}{Lemma}
\newtheorem{proposition}{Proposition}
\newtheorem{theorem}{Theorem}
\newtheorem*{theorem*}{Theorem}
\newtheorem{corollary}{Corollary}

\begin{document}

\title{\textbf{The discrepancy between min-max statistics \\ of Gaussian and Gaussian-subordinated matrices}}
\author{Giovanni Peccati and Nicola Turchi\thanks{\textit{Department of Mathematics, University of Luxembourg.} \\ \hspace*{15pt} The authors are supported by
the FNR grant FoRGES (R-AGR-3376-10).}}
\date{\today}

\maketitle 

\begin{abstract}
We compute quantitative bounds for measuring the discrepancy between the distribution of two min-max statistics involving either pairs of Gaussian random matrices, or one Gaussian and one Gaussian-subordinated random matrix. In the fully Gaussian setup, our approach allows us to recover quantitative versions of well-known inequalities by Gordon (1985, 1987, 1992), thus generalising the quantitative version of the Sudakov-Fernique inequality deduced in Chatterjee (2005). On the other hand, the Gaussian-subordinated case yields generalizations of estimates by Chernozhukov et al. (2015) and Koike (2019). As an application, we establish fourth moment bounds for matrices of multiple stochastic Wiener-It\^o integrals, that we illustrate with an example having a statistical flavour. 

\smallskip

\noindent\textbf{Keywords}: Min-max Statistics; Random Matrices; Gaussian Vectors; Gaussian Fields; Gaussian Analysis; Probabilistic Approximations; Malliavin Calculus.

\smallskip

\noindent\textbf{AMS 2020 Classification}: 60F05; 60G15; 60G70; 60H05; 60H07.

\end{abstract}

\section{Introduction}
\subsection{Overview of our contributions}

In the theory of Gaussian processes an important role is played by inequalities of the \textit{Sudakov-Fernique type}. These results consist in comparisons between extremal value statistics of two distinct Gaussian objects, for example, the maxima of two Gaussian random vectors with different variances. The classical Sudakov-Fernique inequality states that, if \((X_1,\ldots,X_n)\) and \((Y_1,\ldots,Y_n)\) are centered Gaussian random vectors such that \(\Ex\bigl((X_i-X_j)^2\bigr)\le\Ex\bigl((Y_i-Y_j)^2\bigr)\) for all pairs of indices, then
\begin{equation}
	\label{eq:SF}
	\Ex\Bigl(\max_{i\in\{1,\ldots,n\}}X_i\Bigr)\le \Ex\Bigl(\max_{i\in\{1,\ldots,n\}} Y_i\Bigr).
\end{equation}
The inequality \eqref{eq:SF} first appeared in the works of Sudakov \cite{S1,S2} and Fernique \cite{F1},
and a proof is also due to Alexander \cite{A}. Vitale \cite{V} was able to remove the zero-mean assumption with the weaker condition that \(\Ex\bigl((X_1,\ldots, X_n)\bigr)=\Ex\bigl((Y_1,\ldots, Y_n)\bigr)\). Similar inequalities have been studied by Gordon \cite{G1,G2} and Kahane \cite{Ka} in the more general setting of higher-dimensional tensors \((X_{i_1,i_2,\ldots, i_d})\), where the minimum of the classic Sudakov-Fernique inequality is replaced by quantities of the type  \(\min_{i_1}\max_{i_2}\cdots X_{i_1,i_2,\ldots, i_d}\). In \cite{G3} Gordon also studied the comparison between the sums of the first \(k\) ordered statistics of two Gaussian random vectors.  See e.g. \cite{AT07, NPV14}, and the references therein, for a sample of applications of estimates directly related to \eqref{eq:SF} --- ranging from the geometry of Gaussian fields, to stochastic differential equations and statistical mechanics. 

Whereas the aforementioned results are mostly qualitative, in the reference \cite{C} one can find a quantitative counterpart to \eqref{eq:SF}, using integration by parts formulas (see also \cite[Section 2.3]{AT07}). More precisely, in \cite{C} it is established that, if the two Gaussian vectors \(X\) and \(Y\) have the same mean, then
\begin{equation}
	\label{e:chatt}
	\abs[\Big]{\Ex\Bigl(\max_{i\in\{1,\ldots,n\}}X_i\Bigr)-\Ex\Bigl(\max_{i\in\{1,\ldots,n\}} Y_i\Bigr)}\le\sqrt{\max_{i,j}\abs[\big]{\Ex\bigl((X_i-X_j)^2\bigr)-\Ex\bigl((Y_i-Y_j)^2\bigr)}\log n}.
\end{equation}
In the first part of the present work, we extend the study of quantitative bounds of the type \eqref{e:chatt} to the setting considered by Gordon \cite{G1,G2,G3} of min-max statistics of Gaussian random matrices. To motivate the reader, we report below one of our principal contributions on the matter --- see  \Cref{sec:SF} for a full statement and for its proof.

\begin{theorem*}
	Let \((X_{i_1,i_2})\) and \((Y_{j_1,j_2})\) be two $n\times m$ Gaussian random matrices with the same expectation. Then 
	\[
	\begin{split}
		\abs[\Big]{\Ex\bigl(\min_{i_1}\max_{i_2}X_{i_1,i_2}\bigr) - \Ex\bigl(\min_{i_1}\max_{i_2}Y_{i_1,i_2}\bigr)}
		\le&\sqrt{\max_{i_1,i_2,j_1,j_2}\abs[\big]{\Ex\bigl((X_{i_1,i_2}-X_{j_1,j_2})^2\bigr)-\Ex\bigl((Y_{i_1,i_2}-Y_{j_1,j_2})^2\bigr)}} \\
		&\times\Bigl[ \sqrt{\Bigl(1 -\frac{1}{n}\Bigr)\Bigl(2-\oneover{m}\Bigr)\log n}+\sqrt{\Bigl(1-\frac{1}{m}\Bigr)\log m}\Bigr].
			\end{split}
	\]
\end{theorem*}
We will see that our techniques also allow one to recover as special cases virtually all comparison statements for min-max statistics (and their generalizations, like sums of order statistics) proved in \cite{G1,G2,G3}. In particular, an interesting application of our findings is the comparison between the order statistics of two Gaussian random vectors, once they are regarded as min-max of particular Gaussian random matrices. For instance, we can show that
\[
	\abs[\big]{\Ex(X_{(n-1)})-\Ex(Y_{(n-1)})}\le(\sqrt{2}+1)\sqrt{\max_{i,j}\abs[\big]{\Ex\bigl((X_i-X_j)^2\bigr)-\Ex\bigl((Y_i-Y_j)^2\bigr)}\log n},
\]
where the index \((n-1)\) indicates the second maximum coordinate of a \(n\text{-dimensional}\) vector. See Corollary \ref{c:order} below.

It is apparent that bounds such as the ones described above, involving only first moments of extremal statistics, cannot completely describe the similarity between the distributions of the involved quantities. To overcome this shortcoming, Chernozhukov, Chetverikov, and Kato have established in references \cite{CCK15,CCK16} (which crucial installments of the so-called \textit{CCK theory}) bounds on the \textit{Kolmogorov distance} between the laws of the maxima of two Gaussian random vectors, so as to give a more precise description of their closeness. In order to achieve their results, the authors employ some novel anti-concentration inequalities for the maximum statistic of a Gaussian random process. These contributions have been recently extended by Koike in \cite{K} --- to which we refer the reader for a more comprehensive overview of the CCK theory --- where bounds are established on the discrepancy between the maxima of a Gaussian random vector and a smooth Gaussian-subordinated random element.

In \Cref{sec:D12} (see, in particular, Theorem \ref{thm:main}) we generalize some of the results from \cite{CCK15,CCK16, K} to the aforementioned setting of min-max statistics of random matrices: in particular, we derive a bound for the Kolmogorov distance between the laws of the min-max statistics of two random matrices, one of which is Gaussian. In order to do so, we need to recover some new anti-concentration inequalities suitable for our purposes; see for instance Proposition \ref{prop:anticoncentration} and Lemma \ref{lem:anti-con2} below.

One important by-product of our findings are estimates involving matrices of \textit{multiple Wiener-It\^o integrals} (see e.g. \cite[Chapter 2]{NP}), to which we will devote Section \ref{s:mwi}. As an example of application of such estimates, suppose that every entry \((i_1,i_2)\) of an \(n\times m\) matrix is given by the following random quadratic form
\[
	F_{i_1,i_2}=\sum_{u,v=1}^d A_{i_1,i_2}(u,v)\xi_{u}\xi_v-\Ex\Bigl(\sum_{u,v=1}^d A_{i_1,i_2}(u,v)\xi_{u}\xi_v\Bigr),
\]
where \(A_{i_1,i_2}(\cdot ,\cdot)\) is a real-valued symmetric matrix for all \((i_1,i_2)\) and \(\xi_1,\ldots,\xi_d\) is a \(d\text{-dimensional}\) Gaussian random vector. If \((X_{i_1,i_2})\) is a \(n\times m\) centered Gaussian random matrix with the same covariance structure as \((F_{i_1,i_2})\), then one has that
\[
d_{\text{Kol}}(\min_{i_1}\max_{i_2} F_{i_1,i_2},\min_{i_1}\max_{i_2} X_{i_1,i_2})\le C \max_{i_1,i_2}\bigl(\Ex( F^4_{i_1,i_2})-3\Ex(F_{i_1,i_2}^2)^2\bigr)^{1/6} n^{2/3}(\log m)^{1/3}(\log nm)^{2/3}.
\]
where \(C>0\) is an absolute constant and $d_{\text{Kol}}(U,V)$ stands for the Kolmogorov distance between the distribution of the random variables $U,V$ (see \cite[Appendix C]{NP}). An illustration of these findings --- inspired by the statistical theory developed in \cite{K} --- is presented in \Cref{sec:stat}.

\subsection{Notation}\label{ss:notation}
For \(m\in\N\), we write \([m]\) to indicate the sets of integers \(\{1,\ldots,m\}\). 
For \(k\in\onetom\), \(p\in[k]\), \(\set{a_1,\ldots,a_p}\subseteq[m]\) and \(\set{b_1,\ldots,b_q}\subseteq[m]\setminus\set{a_1,\ldots,a_p}\), we define the sets
\begin{align*}
	\kset
	&\coloneqq\set{L\subseteq 2^{\onetom}:\card L= k},\\
	\kset^{a_1\ldots a_p\, \hat b_1\ldots \hat b_q}
	&\coloneqq\set{L\subseteq 2^{\onetom}:\card L= k, \set{a_1,\ldots,a_p}\subseteq L\text{ and } \set{b_1,\ldots,b_q}\subseteq 2^{[m]} \setminus L}.
\end{align*}
Note that \(\card{\kset}=\binom{m}{k}\) and \(\card{\kset^{a_1\ldots a_p\, \hat b_1\ldots \hat b_q}}=\binom{m-p-q}{k-p-q}\).

For \( z=(z_1,\ldots,z_d)\in\R^d\), let \(z_{k}\in\set{z_1,\ldots,z_d}\) be the \(k\text{-th} \) ordered statistics of its components, i.e. \[\min_{i\in[d]} z_i=z_{(1)}\le\ldots\le z_{(k)}\le \ldots z_{(d)}=\max_{i\in[d]} z_i.
\]
If \(z=(z_{i_1,i_2})_{(i_1,i_2)\in\oneton\times\onetom}\in\R^{n\times m}\), we write its \(i_1\text{-th}\) row as \(z_{i_1,\cdot}=(z_{i_1,1},\ldots,z_{i_1,m})\in\R^m\). In particular  \(z_{i_1,(k)}\) indicates the \(k\text{-th}\) ordered statistics of the vector \(z_{i_1,\cdot}\) and \(z_{\cdot, (k)}\) stands for the vector \((z_{1,(k)},\ldots, z_{n,(k)})\in\R^n\). Throughout the paper, we will refer to the quantity
\[
	\min\max z\coloneqq\min_{i_1\in [n]} z_{i_1, (m)} = \min_{i_1\in [n]} \max_{i_2\in [m]} z_{i_1, i_2}
\]	
as the \textit{min-max statistic} of the matrix $z$.  We will always work on a fixed probability space \((\Omega,\mathcal{F},\Prob)\) and write \(\Ex\) for the expectation with respect to \(\Prob\).

\section{Comparison of min-max statistics for two Gaussian random matrices}
\label{sec:SF}

\subsection{Main estimates}

The forthcoming statement is one of the main contributions of the present work, containing as special cases several results evoked in the Introduction; in particular, the inequalities \eqref{eq:SF}--\eqref{e:chatt} correspond to the case $n=k=1$, $m\geq 1$ of our result; Theorem 1.4 in \cite{G1} corresponds to the case $n,m\geq 1$ and $k=1$; Theorem 1.3 in \cite{G3} corresponds to the choice $n=1$, $m\geq 1$ and $k\leq m$ --- see the subsequent discussion.

\begin{theorem}
	\label{thm:bound}
	Let \(X=(X_{i_1,i_2})_{(i_1,i_2)\in\oneton\times\onetom}\) and \(Y=(Y_{i_1,i_2})_{(i_1,i_2)\in\oneton\times\onetom}\) be two Gaussian random matrices with \(\Ex(X_{i_1,i_2})=\Ex(Y_{i_1,i_2})\) for every \((i_1,i_2)\in\oneton\times\onetom\).  Define \(\gamma^X_{i_1,i_2;j_1,j_2}\coloneqq\Ex(( X_{i_1,i_2}-X_{j_1,j_2})^2)\), \(\gamma^Y_{i_1,i_2;j_1,j_2}\coloneqq\Ex(( Y_{i_1,i_2}-Y_{j_1,j_2})^2)\) and let 
	\[
		\gamma\coloneqq\max\limits_{\subalign{(i_1,i_2)&\in\oneton\times\onetom\\(j_1,j_2)&\in\oneton\times\onetom}}\abs[\big]{\gamma^X_{i_1,i_2;j_1,j_2}-\gamma^Y_{i_1,i_2;j_1,j_2}}.
	\] 
	Then, for all $k\in[m]$,
	\begin{equation}\label{e:mainz}
	\begin{split}
		\abs[\Big]{\Ex\Big(\min_{i_1\in\oneton}\sum_{h=m-k+1}^m X_{i_1,(h)}\Bigr) &- \Ex\Big(\min_{i_1\in\oneton}\sum_{h=m-k+1}^m Y_{i_1,(h)}\Bigr)}\\
		&\le\sqrt{\gamma k} \cdot \Bigl[ \sqrt{k\Bigl(1 -\frac{1}{n}\Bigr)\Bigl(2-\oneover{m}\Bigr)\log n}+\sqrt{\Bigl(1-\frac{k}{m}\Bigr)\log\binom{m}{k}}\Bigr].
			\end{split}
	\end{equation}
	Moreover,  if, for every \((i_2,j_2)\in\oneton\times\onetom\),
	\[
	\begin{cases}
	\gamma^X_{i_1,i_2;j_1,j_2}\le \gamma^Y_{i_1,i_2;j_1,j_2} &\quad \text{ if } i_1=j_1 \\
	\gamma^X_{i_1,i_2;j_1,j_2}\ge \gamma^Y_{i_1,i_2;j_1,j_2} &\quad \text{ if } i_1\neq j_1,
	\end{cases}
	\]
	then
	\[
	\Ex\Big(\min_{i_1\in\oneton}\sum_{h=m-k+1}^m X_{i_1,(h)}\Bigr)\le \Ex\Big(\min_{i_1\in\oneton}\sum_{h=m-k+1}^m Y_{i_1,(h)}\Bigr).
	\]
\end{theorem}


\medskip

\noindent\textit{Remark.} In order to substantiate the claims preceding the statement of Theorem \ref{thm:bound}, we put forward the following two special cases: (i) when $n=1$, then $X$ and $Y$ are $m$-dimensional Gaussian vectors, and the quantities inside the expectations on the left-hand side of \eqref{e:mainz} are the sums of the order statistics of orders $m-k+1$ up to $m$ of $X$ and $Y$ (in particular, when $k=1$ we recover the maxima); (ii) when $k=1$ and no restrictions are put on $n,m$ then the random variables on the left-hand side of \eqref{e:mainz} are the min-max statistics of $X$ and $Y$.

\medskip

\noindent\textit{Remark.} There is no conceptual obstacle in extending \Cref{thm:bound} to the more general case of a \(d\text{-dimensional}\) Gaussian tensor \((X_{i_1,\dots,i_d})_{(i_1,\dots,i_d)\in[n_1]\times\cdots\times[n_d]}\) and the investigation of the quantity (say \(d\) is even without loss of generality and \(k_i\le n_i\) for all \(i\le d\))
\[
\sum_{h_1=1}^{k_1}\sum_{h_2=n_2-k_2+1}^{n_2}\cdots\sum_{h_{d-1}=1}^{k_{d-1}}\sum_{h_d=n_d-k_d+1}^{n_d} (\cdots(X_{\cdot,\ldots,\cdot,(h_d)})_{\cdot,\ldots,\cdot,(h_{d-1})})\cdots)_{(h_1)},
\]
but we decided not to perform it explicitly, in order to keep the length of the paper within bounds.

\medskip

One remarkable consequence of \Cref{thm:bound} is that it yields comparison criteria for the expected values of order statistics associated with Gaussian random vectors.

\begin{corollary}\label{c:order}
	Let \(W=(W_i)_{i\in[d]}\) and \(Z=(Z_i)_{i\in[d]}\) be two Gaussian random vectors with \(\Ex(W_i)=\Ex(Z_i)\) for every \(i\in[d]\), and let \(\gamma=\max_{(i,j)\in[d]^2} \abs{\Ex\bigl((W_i-W_j)^2\bigl)-\Ex\bigl((Z_i-Z_j)^2\bigl)}\). Then, for any \(h\in[d]\),
		\[
			\abs[\big]{\Ex (W_{(h)}) -\Ex (Z_{(h)}) }\le \sqrt{\gamma}\Bigl(\sqrt{2\log\binom{d}{h}}+\sqrt{\log h}\Bigr)\le \sqrt{\gamma}\Bigl(\sqrt{2h(1+\log (d/h))}+\sqrt{\log h}\Bigr).
		\]
\end{corollary}

\begin{proof}
	The key idea is that \(W_{(h)}\) (respectively, \(Z_{(h)}\)) is the min-max statistic (see Section \ref{ss:notation}) of a matrix \(X\) with \(\binom{d}{h}\) rows, where each row corresponds to a distinct subset of \(W\) (respectively, \(Z\)) with cardinality \(h\) (the order of the elements of the subset within a single row is immaterial). To see this, observe that the rows of the matrix $X$ described above are such that: (i) there exists at least one row containing $W_{(h)}$ as a maximal element, and (ii) every other row contains one element that is larger or equal to $W_{(h)}$.
	%
	Using now \Cref{thm:bound} with \(n=\binom{d}{h}\), \(m=h\) and \(k=1\) yields the first bound. The second bound follows easily from the first noting that \(\binom{d}{h}\le(\frac{ed}{h})^h\).
\end{proof}


\medskip

\noindent\textit{Remark.} An alternate class of local comparison theorems for (vectors of) order statistics of Gaussian matrices can be found in \cite{DHJL} -- see the discussion following Theorem \ref{thm:main} below for further details.

\medskip

\noindent\textit{Remark.} We now show that, when \(m\) is fixed, the bound of \Cref{thm:bound} is sharp in the order of \(n\) and \(k\). First, let \(Y\equiv 0\) and \(X\) be a matrix with \(m\) columns which are the copy of a same \(n\text{-dimensional}\) standard Gaussian vector \(X'\). Then \(\gamma=2\) for every \(n\) and 
\[
	\sum_{h=m-k+1}^m X_{i_1,(h)}=k X'_{i_1}.
\]
{It is known from extreme value theory that, as \(n\) diverges, the expectation of \(\min_{i_1} X'_{i_1}\) is of order \(\sqrt{\log n}\) (up to constants). Then, the expectation of \(\sum_{h=m-k+1}^m X_{i_1,(h)}\) is of order \(k\sqrt{\log n}\), matching the order of the bound \eqref{e:mainz} in this specific case.}

Analogously, when \(n\) is fixed, then the  bound is sharp in the order of \(m\) and \(k\) in the regime where \(k, m-k\ll m\). To see this, let \(Y\equiv 0\) and \(X\) be a matrix with \(n\) rows which are the copy of a same \(m\text{-dimensional}\) (transposed) Gaussian vector \( X'\). This time suppose without loss of generality that \(m\) is a multiple of \(k\), \(m=\tilde m k\), \(\tilde m\in \N\) and that \(X'\) is the collection of \(k\) copies of the same standard \(\tilde m\text{-dimensional}\) Gaussian vector \(\tilde X\).  
{Note that in this case
\[
\sum_{h=m-k+1}^m X_{i_1,(h)}=k  \max_{j\in[\tilde m]}\tilde X_{j}
\]
for all \(i_1\in[n]\). However, when \(\tilde m\) diverges, the expected value of \(\max_{j\in[\tilde m]}\tilde X_{j}\) is of order \(\sqrt{\log \tilde m}=\sqrt{\log (m/k)}\), and, accordingly, the expectation of \(k  \max_{j\in[\tilde m]}\tilde X_{j}\) is of order
\[
	k\sqrt{\log (m/k)}=\sqrt{k}\sqrt{\log (m/k)^k}\approx \sqrt{k}\sqrt{\log\binom{m}{k}},
\]
where the last approximation holds in the aforementioned regime of \(k\) with respect to \(m\). Again, we recover asymptotically the bound \eqref{e:mainz}.
}

\medskip

The next section contains six technical results that are pivotal in the proof of Theorem \ref{thm:bound}.

\subsection{Six ancillary lemmas}

\begin{lemma}
	\label{lem:unifbound}
	For \(\beta, \delta >0 \) and $k\in [m]$,  define the function \(	f_k^{\beta,\delta}\colon\R^{n\times m}\to\R\) by
	\[
	f_k^{\beta,\delta}(x)\coloneqq-\oneover{\beta\delta}\log\sum_{\ell_1=1}^n\Bigl(\sum\limits_{L\in\compset}\exp\Bigl(\beta\sum\limits_{\ell_2\in L}x_{\ell_1,\ell_2}\Bigr)\Bigr)^{-\delta}.
	\]
	For every \(x\in\R^{n\times m}\), one has that 
	\[
	f_k^{\beta,\delta}(x)-\oneover{\beta}\log\binom{m}{k}\le\Bigl(\sum_{h=m-k+1}^m x_{\cdot,(h)}\Bigr)_{(1)}\le	f_k^{\beta,\delta}(x)+\oneover{\beta\delta}\log n.
	\]
\end{lemma}
\begin{proof}
	Let \(d\in\N\) and \(z\in\R^d\). When \(\beta>0 \) the following inequality holds
	\[
		\oneover{d}\sum_{i=1}^d e^{\beta z_i}\le e^{\beta z_{(d)}}\le \sum_{i=1}^d e^{\beta z_i},
	\]
	in particular 
	\begin{equation}
	\label{eq:max}
		\oneover{\beta}\log\sum_{i=1}^d e^{\beta z_i}-\frac{\log d}{\beta}\le z_{(d)}\le	\oneover{\beta}\log \sum_{i=1}^d e^{\beta z_i}.
	\end{equation}
	Similarly for the minimum instead, it holds that, for every \(\beta'>0\),
	\begin{equation}
	\label{eq:min}
	-\oneover{\beta'}\log\sum_{i=1}^d e^{-\beta' z_i}\le z_{(1)}\le	-\oneover{\beta'}\log \sum_{i=1}^d e^{-\beta' z_i}+\frac{\log d}{\beta'}.
	\end{equation}
	For each \(\ell_1\in\{1,\ldots, n\}\), consider the vector \(z=\bigl(\sum_{\ell_2\in L} x_{\ell_1,\ell_2}\bigr)_{L\in\compset}\in\R^{\binom{m}{k}}\)
	and apply \eqref{eq:max} to it. Notice that  \(z_{(\binom{m}{k})}=\sum\limits_{h=m-k+1}^m x_{\ell_1,(h)}\). We get
	\[
\oneover{\beta}\log \sum_{L\in\compset} \exp\Bigl(\beta\sum\limits_{\ell_2\in L}x_{\ell_1,\ell_2}\Bigr)-\oneover{\beta}\log\binom{m}{k}	\le\sum\limits_{h=m-k+1}^m x_{\ell_1,(h)}\le \oneover{\beta}\log \sum_{L\in\compset} \exp\Bigl(\beta\sum\limits_{\ell_2\in L}x_{\ell_1,\ell_2}\Bigr).
	\]
Now we want to isolate the minimum of the vector \(	\Bigl(\sum\limits_{h=m-k+1}^m x_{\cdot,(h)}\Bigr)\in\R^n\), and we use \eqref{eq:min} to do so.
Notice that, for \(\beta'=\beta\delta\), we get
\[
-\oneover{\beta'}\log\sum\limits_{i=1}^n \exp\Bigl(-\beta'\cdot  \oneover{\beta}\log \sum_{L\in\compset} \exp\Bigl(\beta\sum\limits_{\ell_2\in L}x_{\ell_1,\ell_2}\Bigr)\Bigr)=	f_k^{\beta,\delta}(x)\]
 and that
 \[
\begin{split}
-\oneover{\beta'}\log&\sum\limits_{i=1}^n \exp\Bigl(-\beta'\cdot  \Bigl(\oneover{\beta}\log \sum_{L\in\compset} \exp\Bigl(\beta\sum\limits_{\ell_2\in L}x_{\ell_1,\ell_2}\Bigr)-\oneover{\beta}\log\binom{m}{k}\Bigr)\Bigr)\\
&
=-\oneover{\beta'}\log\Bigr[\binom{m}{k}^\delta\sum\limits_{i=1}^n \exp\Bigl(-\beta\cdot  \Bigl(\oneover{\beta}\log \sum_{L\in\compset} \exp\Bigl(\beta\sum\limits_{\ell_2\in L}x_{\ell_1,\ell_2}\Bigr)\Bigr)\Bigr)\Bigr]=	f_k^{\beta,\delta}(x)-\oneover{\beta}\log\binom{m}{k},
\end{split}
\]
which concludes the proof by monotonicity.
\end{proof}

For \(h\in\{0,\ldots,m\}\), let \(A\subseteq\onetom\) with \(\card{A}=h \). If \(h>0\) we write \(A=\set{a_1,\ldots,a_h}\) with \(a_1<\ldots <a_h\).
It is convenient at this point to define the functions, for any \(i_1\in\oneton\), \(p_{i_1}^A,q_{i_1}\colon\R^n\to[0,+\infty)\) by
\begin{align*}
p_{i_1}^A(x)=p^{a_1,\ldots,a_h}_{i_1}(x)&\coloneqq\dfrac{\sum\limits_{L\in\compset^{a_1\cdots a_h}}\!\!\!\!\!\exp\bigl(\beta\sum\limits_{\ell\in L}x_{i_1,\ell}\bigr)}{\sum\limits_{L\in\compset}\!\!\exp\bigl(\beta\sum\limits_{\ell\in L}x_{i_1,\ell}\bigr)},\\
q_{i_1}(x)&\coloneqq\sum_{\ell_1=1}^n\Biggl(\frac{\sum\limits_{L\in\compset}\!\!\exp\bigl(\beta\sum\limits_{\ell_2\in L}x_{i_1,\ell_2}\bigr)}{\sum\limits_{L\in\compset}\!\!\exp\bigl(\beta\sum\limits_{\ell_2\in L}x_{\ell_1,\ell_2}\bigr)}\Biggr)^{\delta}.
\end{align*}
where we use the notational conventions \(p_{i_1}^\emptyset\equiv 1\) and \(\sum_{L\in\emptyset}=0\).
\begin{lemma}
		\label{lem:sums}
Let \(b_1<\ldots< b_h\) and \(B=\set{a_{b_1},\ldots,a_{b_h}}\subseteq A\). For every \(x\in\R^n \), one has that
\[
\sum_{a_{b_1},\ldots,a_{b_h}=1}^m p^A_{i_1}(x)=k^h p^{A\setminus B}_{i_1}(x).
\]
and
\[
\sum_{i_1=1}^n \oneover{q_{i_1}(x)}=1.
\]
In particular, for \(A=B\) we infer that
\[\sum_{a_1,\ldots,a_h=1}^m p^{a_1,\ldots,a_h}_{i_1}(x)=k^h.\]
Notice that since \(p_{i_1}^{a_1,a_1}=p_{i_1}^{a_1}\), this implies also that
\[
\sum_{\substack{a_2=1\\a_1\neq a_2}}^m p^{a_1,a_2}_{i_1}(x)=(k-1)p_{i_1}^{a_1}.
\]
\end{lemma}

\begin{proof}
	We only prove the first equation for \(h=1\), the general case is then similarly proven by iteration. One has that
	\[
			\begin{split}
			\sum_{a=1}^m \sum\limits_{L\in\compset^a} \exp\Bigl(\beta \sum_{\ell\in L} x_{i_1,\ell} \Bigr)&=	\sum_{a=1}^m \sum\limits_{L\in\compset} \indic{\{a\in L\}} \exp\Bigl(\beta \sum_{\ell\in L} x_{i_1,\ell} \Bigr) \\
			&=\sum\limits_{L\in\compset} \exp\Bigl(\beta \sum_{\ell\in L} x_{i_1,\ell} \Bigr)\sum_{a=1}^m\indic{\{a\in L\}}\\
			&=k\sum\limits_{L\in\compset} \exp\Bigl(\beta \sum_{\ell\in L} x_{i_1,\ell} \Bigr),
			\end{split}
	\]
	that gives the claim. The second equation is a straightforward consequence of the definition of \(q_{i_1}\).
\end{proof}

\begin{lemma}
	\label{lem:0sum}
		For all \((i_1,i_2)\in\oneton\times\onetom \),
	\begin{equation}
	\label{eq:sum0}
	\sum_{(j_1,j_2)\in\oneton\times\onetom}\frac{\partial^2 f_k^{\beta,\delta}}{\partial x_{i_1,i_2}\partial x_{j_1,j_2}}\equiv 0.
	\end{equation}
\end{lemma}

\begin{proof}
	A direct computation shows that
		\begin{equation}
	\label{eq:firstpartial}
	\frac{\partial f_k^{\beta,\delta}}{\partial x_{i_1,i_2}}=\frac{p_{i_1}^{i_2}}{q_{i_1}}
	\end{equation}
	and that
		\begin{equation}
	\label{eq:secondpartial}
	\frac{\partial^2 f_k^{\beta,\delta}}{\partial x_{i_1,i_2}\partial x_{j_1,j_2}}
	=
	\beta\Bigl[\delta\frac{p_{i_1}^{i_2}p_{j_1}^{j_2}}{q_{i_1}q_{j_1}}+\frac{\delta_{i_1,j_1}}{q_{i_1}}\bigl(-(1+\delta)p_{i_1}^{i_2}p_{j_1}^{j_2}+ p_{i_1}^{i_2,j_2}\bigr)\Bigr].
	\end{equation}
	We will evaluate the contribution of the three summands of \eqref{eq:secondpartial} separately and show that they balance out to 0. Notice that \(\beta/q_{i_1}\) is a common multiplicative factor to all three so we can ignore it. We use both properties stated in \Cref{lem:sums} to deduce that
	\[
	\sum_{(j_1,j_2)\in[n]\times[m]}^n\delta\,\frac{p_{i_1}^{i_2}p_{j_1}^{j_2}}{q_{j_1}}=\delta k p_{i_1}^{i_2},
	\]
	\[
	\sum_{(j_1,j_2)\in[n]\times[m]}^n-(1+\delta)p_{i_1}^{i_2}p_{i_1}^{j_2}=-(1+\delta)k p_{i_1}^{i_2},
	\]
	\[
	\sum_{(j_1,j_2)\in[n]\times[m]}^n p_{i_1}^{i_2,j_2}=k p_{i_1}^{i_2  },
	\]
which concludes the proof.
\end{proof}

	\begin{lemma}
		\label{lem:Gordontrick}
		Under the above notation and assumptions, one has that
	\begin{equation*}
	\begin{split}
			\sum_{\substack{(i_1,i_2)\in[n]\times[m]\\(j_1,j_2)\in[n]\times[m]}}	&\frac{\partial^2 f^{\beta,\delta}_k}{\partial x_{i_1,i_2}\partial x_{j_1,j_2}}(\sigma^Y_{i_1,i_2;j_1,j_2}-\sigma^X_{i_1,i_2;j_1,j_2})\\
			=
			&-\oneover{2}	\sum_{\substack{(i_1,i_2)\in[n]\times[m]\\(j_1,j_2)\in[n]\times[m]\\(i_1,i_2)\neq (j_1,j_2)}}		\frac{\partial^2 f^{\beta,\delta}_k}{\partial x_{i_1,i_2}\partial x_{j_1,j_2}}(\gamma^Y_{i_1,i_2;j_1,j_2}-\gamma^X_{i_1,i_2;j_1,j_2}).
	\end{split}
	\end{equation*}
\end{lemma}
	\begin{proof}
		First, note that \(\gamma^X_{i_1,i_2;j_1,j_2}=\sigma^X_{i_1,i_2;i_1,i_2}-2\sigma^X_{j_1,j_2;j_1,j_2}+\sigma^X_{i_1,i_2;i_1,i_2}\), analogously for \(Y\). We can split the first sum of the statement as 
		\[
			\begin{split}
			\sum_{\substack{(i_1,i_2)\in[n]\times[m]\\(j_1,j_2)\in[n]\times[m]}}&\frac{\partial^2 f^{\beta,\delta}_k}{\partial x_{i_1,i_2}\partial x_{j_1,j_2}}(\sigma^Y_{i_1,i_2;j_1,j_2}-\sigma^X_{i_1,i_2;j_1,j_2})\\
			&=\oneover{2}\sum_{(i_1,i_2)\in[n]\times[m]}\frac{\partial^2 f_k^{\beta,\delta}}{\partial x_{i_1,i_2}\partial x_{i_1,i_2}}\bigl(\sigma^Y_{i_1,i_2;i_1,i_2}-\sigma^X_{i_1,i_2;i_1,i_2}\bigr)\\
			&-\oneover{2}\sum_{\substack{(i_1,i_2)\in[n]\times[m]\\(j_1,j_2)\in[n]\times[m]\\(i_1,i_2)\neq (j_1,j_2)}}		\frac{\partial^2 f^{\beta,\delta}_k}{\partial x_{i_1,i_2}\partial x_{j_1,j_2}}(-2\sigma^Y_{i_1,i_2;j_1,j_2}+2\sigma^X_{i_1,i_2;j_1,j_2})\\
			&+\oneover{2}\sum_{(j_1,j_2)\in[n]\times[m]}\frac{\partial^2 f_k^{\beta,\delta}}{\partial x_{j_1,j_2}\partial x_{j_1,j_2}}\bigl(\sigma^Y_{j_1,j_2;j_1,j_2}-\sigma^X_{j_1,j_2;j_1,j_2}\bigr).
			\end{split}
		\]
		By \Cref{lem:0sum} we know that 
			for all \((h_1,h_2)\in\oneton\times\onetom \),
		\[
			\frac{\partial^2 f_k^{\beta,\delta}}{\partial x_{h_1,h_2}\partial x_{h_1,h_2}}=
			-\sum_{\substack{(j_1,j_2)\in\oneton\times\onetom\\(j_1,j_2)\neq(h_1,h_2)}}\frac{\partial^2 f_k^{\beta,\delta}}{\partial x_{i_1,i_2}\partial x_{j_1,j_2}},
		\]
		which allows us to conclude.
	\end{proof}

\begin{lemma}
	\label{lem:signs}
	For all \(x\in\R^{n\times m}\),
		\[
\frac{\partial^2 f^{\beta,\delta}_k}{\partial x_{i_1,i_2}\partial x_{j_1,j_2}}(x)\begin{cases}
\ge 0&\quad\text{ if } i_1\neq j_1\\
\le 0&\quad\text{ if } i_1=j_1  \text{ and } i_2\neq j_2.
\end{cases}
\]
\end{lemma}

\begin{proof}
	Notice that we can rewrite \eqref{eq:secondpartial} in the following way:
\begin{equation}
\label{eq:sumsplit}
\frac{\partial^2 f_k^{\beta,\delta}}{\partial x_{i_1,i_2}\partial x_{j_1,j_2}}=\frac{\beta}{q_{i_1}}\Bigl[\delta p_{i_1}^{i_2}p_{j_1}^{j_2}\Bigl(\oneover{q_{j_1}}-\delta_{i_1,j_1}\Bigr)+\delta_{i_1,j_1}\bigl(p_{i_1}^{i_2,j_2}-p_{i_1}^{i_2}p_{j_1}^{j_2}\bigr)\Bigr].
\end{equation}
Obviously when \(i_1\neq j_1 \) then the expression is positive. When \(i_1= j_1\), notice that \(1/q_{i_1}-1\le 0\) as \(q_{i_1}\ge 1\), being a sum of non-negative summands whose at least one is exactly equal to \(1\) (the one corresponding to \(\ell=i_1\)). It thus remain to prove only that \(p_{i_1}^{i_2,j_2}-p_{i_1}^{i_2}p_{i_1}^{j_2}\le 0\) whenever \(i_2\neq j_2\).
Writing, \(y_\ell\coloneqq e^{\beta x_{i_1,\ell}}>0\) and multiplying  in both sides by \(	\sum_{L\in\LL_k}\exp\bigl(\beta\sum_{\ell\in L}x_{i_1,_\ell}\bigr)	\), this is equivalent to
\[
\sum_{L\in\LL_k}\prod_{\ell\in L}y_\ell	\sum_{L'\in\LL_k^{ij}}\prod_{\ell'\in L'}y_{\ell'}
\le 
\sum_{L\in\LL_k^i}\prod_{\ell\in L}y_\ell\sum_{L'\in\LL_k^{j}}\prod_{\ell'\in L'}y_{\ell'},
\]
where we also renamed \(i_2\) and \(j_2\) as \(i\) and \(j\), respectively, for simplicity of notation.
Since we can decompose the first double sum as 
\[
\sum_{L\in\LL_k^{\vphantom{\hat ij}}}\sum_{L'\in\LL_k^{\vphantom{\hat ij}ij}}=	\sum_{L\in\LL_k^{\vphantom{\hat ij}ij}}\sum_{L\prime\in\LL_k^{\vphantom{\hat ij}ij}}+	\sum_{L\in\LL_k^{i\hat j}}\sum_{L'\in\LL_k^{ij}}+\sum_{L\in\LL_k^{\hat i j}}\sum_{L'\in\LL_k^{ij}}+\sum_{L\in\LL_k^{\hat i\hat j}}\sum_{L'\in\LL_k^{ij}},
\]
and the second as
\[
\sum_{L\in\LL_k}\sum_{L'\in\LL_k^{ij}}=	\sum_{L\in\LL_k^{ij}}\sum_{L'\in\LL_k^{ij}}+	\sum_{L\in\LL_k^{i\hat j}}\sum_{L'\in\LL_k^{ij}}+\sum_{L\in\LL_k^{ i  j}}\sum_{L'\in\LL_k^{\hat i j}}+\sum_{L\in\LL_k^{ i\hat j}}\sum_{L'\in\LL_k^{\hat i j}},
\]
it appears that the only comparison that remains to be checked is
\[
\sum_{L\in\LL_k^{\hat i\hat j}}\sum_{L'\in\LL_k^{ij}}\prod_{\ell\in L}y_\ell	
\prod_{\ell'\in L'}y_{\ell'}\le 	\sum_{L\in\LL_k^{ i\hat j}}\sum_{L'\in\LL_k^{\hat ij}}\prod_{\ell\in L}y_\ell	
\prod_{\ell'\in L'}y_{\ell'}.
\]
By simplifying a factor \(y_i y_j\) on both sides, this is equivalent to
\[
\sum_{L\in\LL_k}\sum_{L'\in\LL_{k-2}}\prod_{\ell\in L}y_\ell	
\prod_{\ell'\in L'}y_{\ell'}\le 	\sum_{L\in\LL_{k-1}}\sum_{L'\in\LL_{k-1}}\prod_{\ell\in L}y_\ell	
\prod_{\ell'\in L'}y_{\ell'},
\]
with the caveat that here \(\mathcal{L} \) indicates a subset out of a total of \(m-2\) indexes (and not \(m\)). We will show that every distinct double product of the LHS appears with a larger multiplicity in the RHS.
Let \(L\) be a multiset of $2k$ indices (i.e. a set with \(2k\) elements that can be repeated, out of a set of  \(m-2\) total indexes).
Then, the quantity \(\prod_{\ell\in {L}} y_{\ell}\) appears in the LHS if and only if the two following conditions are met
\begin{itemize}
	\item no index \(\ell\) appears in \( L\) more than twice;
	\item there are at most \(k-2\) repeated indices.
\end{itemize}
Suppose that exactly \(r\le k-2\) indices are repeated in \(L\). These indexes have to appear both in each \(L\) and \(L'\) of the RHS. Then, in the LHS there are \(\binom{2k-2r-2}{k-r-2}\) ways to rearrange the remaining indexes among \(L\) and \(L'\), while in the RHS there are \(\binom{2k-2r-2}{k-r-1}\) of those. Notice that the latter quantity is a central binomial coefficient so it is necessarily larger than the former one.
\end{proof}
	
		\begin{lemma}
		\label{lem:absbound}
		For all \(x\in\R^{n\times m}\), it holds that
		\[
			\sum_{\subalign{(i_1,i_2)&\in[n]\times[m]\\(j_1,j_2)&\in[n]\times[m]\\(i_1,i_2)&\neq (j_1,j_2)}}\abs[\Big]{\frac{\partial^2 f^{\beta,\delta}_k}{\partial x_{i_1,i_2}\partial x_{j_1,j_2}}}
			\le
			\beta\frac{k}{m} \bigl(m-k+a_n(2m-1)k\delta\bigr),
		\]
	where \(a_n\coloneqq 1-\oneover{n}\in[0,1)\).
	\end{lemma}
	
	\begin{proof}
		Again, we evaluate the contribution of the summands of \eqref{eq:secondpartial} separately. We make use extensively of \Cref{lem:sums}. We start with the case \(i_1\neq j_1 \):
		\[
		\oneover{\beta}\sum_{\subalign{(i_1,i_2)&\in[n]\times[m]\\(j_1,j_2)&\in[n]\times[m]\\i_1&\neq j_1}}		\abs[\Big]{\frac{\partial^2 f^{\beta,\delta}_k}{\partial x_{i_1,i_2}\partial x_{j_1,j_2}}}=	\sum_{\subalign{(i_1,i_2)&\in[n]\times[m]\\(j_1,j_2)&\in[n]\times[m]\\i_1&\neq j_1}}\delta\,\frac{p_{i_1}^{i_2}p_{j_1}^{j_2}}{q_{i_1}q_{j_1}}=\delta\,k^2\Bigl(1-\sum_{i_1\in\oneton}\oneover{q_{i_1}^2}\Bigr)\le\delta k^2\Bigl(1-\oneover{n}\Bigr).
		\]
		The last inequality is due to the fact that the terms \(1/q_{i_1}\) are positive and sum to \(1\), in particular by Cauchy-Schwarz inequality the sum of their squares is minimum when all of them are equal to \(1/n\). 
		
		For the case \(i_1= j_1 \) we make use of the expression \eqref{eq:sumsplit} in which both summands are non-positive, in such a way that we can write
		
		\[
		\begin{split}
		\oneover{\beta}\sum_{\substack{(i_1,i_2)\in[n]\times[m]\\(j_1,j_2)\in[n]\times[m]\\i_1= j_1\\i_2\neq j_2}}	\abs[\Big]{\frac{\partial^2 f^{\beta,\delta}_k}{\partial x_{i_1,i_2}\partial x_{j_1,j_2}}}&
		=
		\sum_{\substack{(i_1,i_2)\in[n]\times[m]\\j_2\in[m]\\i_2\neq j_2}}\frac{1}{q_{i_1}}\Bigl[\delta p_{i_1}^{i_2}p_{i_1}^{j_2}\Bigl(1-\oneover{q_{i_1}}\Bigr)+\bigl(p_{i_1}^{i_2}p_{i_1}^{j_2}-p_{i_1}^{i_2,j_2}\bigr)\Bigr]\\
		&= \sum_{\substack{(i_1,i_2)\in[n]\times[m]\\j_2\in[m]\\i_2\neq j_2}} \Bigl(\frac{1+\delta}{q_{i_1}}-\frac{\delta}{q_{i_1}^2} \Bigr)p_{i_1}^{i_2}p_{i_1}^{j_2}- 		\sum_{\substack{(i_1,i_2)\in[n]\times[m]\\j_2\in[m]\\i_2\neq j_2}}\frac{p_{i_1}^{i_2,j_2}}{q_{i_1}}.
		\end{split}
		\]
		The first sum can be estimated as follows:
		\[
		\begin{split}
		\sum_{\substack{(i_1,i_2)\in[n]\times[m]\\j_2\in[m]\\i_2\neq j_2}}\Bigl(\frac{1+\delta}{q_{i_1}}-\frac{\delta}{q_{i_1}^2} \Bigr)p_{i_1}^{i_2}p_{i_1}^{j_2}
		&=
		\sum_{\substack{(i_1,i_2)\in[n]\times[m]}}\Bigl(\frac{1+\delta}{q_{i_1}}-\frac{\delta}{q_{i_1}^2} \Bigr)p_{i_1}^{i_2}\sum_{\substack{j_2\in[m]\\i_2\neq j_2}}p_{i_1}^{j_2}\\
		&=
		\sum_{\substack{(i_1,i_2)\in[n]\times[m]}}\Bigl(\frac{1+\delta}{q_{i_1}}-\frac{\delta}{q_{i_1}^2} \Bigr)p_{i_1}^{i_2}(k-p_{i_1}^{i_2})\\
			&=
		\sum_{\substack{i_1\in[n]}}\Bigl(\frac{1+\delta}{q_{i_1}}-\frac{\delta}{q_{i_1}^2} \Bigr)\Bigl(k^2-\sum_{i_2\in[m]}(p_{i_1}^{i_2})^2\Bigr)\\
		&\le
		 k^2\Bigl(1-\frac{1}{m}\Bigr)\Bigl(1+\delta\Bigl(1-\oneover{n}\Bigr)\Bigr)
		\end{split}
		\]
		where we used the fact that the numbers \(\{p_{i_1}^{i_2}\}_{i_2\in[m]}\) sum to $k$, so that the minimum of \(\sum_{i_2\in[m]}(p_{i_1}^{i_2})^2\) is \(k^2/m\), again by the Cauchy-Schwarz inequality. Concerning the last summand, we obtain
		\[
		\sum_{\substack{(i_1,i_2)\in[n]\times[m]\\j_2\in[m]\\i_2\neq j_2}}\frac{p_{i_1}^{i_2,j_2}}{q_{i_1}}=\!\!\!\!\sum_{\substack{(i_1,i_2)\in[n]\times[m]}}\frac{1}{q_{i_1}}\sum_{\substack{j_2\in[m]\\i_2\neq j_2}}p_{i_1}^{i_2,j_2}=\!\!\!\!\sum_{\substack{(i_1,i_2)\in[n]\times[m]}}\!\!\!\!(k-1)\frac{p_{i_1}^{i_2}}{q_{i_1}}=k(k-1).
		\]
		Note that the sums over indexes \((i_1\neq j_1)  \) (respectively, \(i_2\neq j_2\)) make sense only when \(n>1\) (respectively, \(m>1\)) but those expressions are \(0\) anyway in such cases. Merging the three contributions, one deduces the desired conclusion.
	\end{proof}
	
We are now ready to prove Theorem \ref{thm:bound}.
	
	\subsection{Proof of \texorpdfstring{\Cref{thm:bound}}{Theorem 2}}
		Once the analytical lemmas presented in the previous section are established, the proof follows from a classical interpolation technique --- already exploited e.g. in \cite{C, NPV14} or \cite[Chapter 6]{NP}. Without loss of generality, we can assume that \(X\) and \(Y\) are independent. Let \(\mu_{i_1,i_2}=\Ex(X_{i_1,i_2})=\Ex(X_{i_1,i_2})\). For \(t\in[0,1]\) we consider the random interpolation matrix \(Z_t\in\R^{n\times m}\) whose entries are given by
		\[
			(Z_t)_{i_1,i_2}\coloneqq\sqrt{1-t}(X_{i_1,i_2}-\mu_{i_1,i_2})+\sqrt{t}(Y_{i_1,i_2}-\mu_{i_1,i_2})+\mu_{i_1,i_2}.
		\]
		Note that \(Z_0=X\), \(Z_1=Y\) and \(\Ex((Z_t)_{i_1,i_2})=\mu_{i_1,i_2}\) for all \(t\in[0,1]\) and all \({i_1,i_2}\in[n]\times[m]\).
		We also define \(\psi(t)\coloneqq \Ex(f^{\beta,\delta}_k(Z_t))\), in such a way that \(\psi\) is differentiable with derivative
		\[
			\psi'(t)=\oneover{2}\sum_{(j_1,j_2)\in[n]\times[m]}\Ex\Bigl(\frac{\partial  f^{\beta,\delta}_k}{\partial x_{j_i,j_2}}(Z_t)\Bigl(\frac{Y_{j_1,j_2}-\mu_{j_1,j_2}}{\sqrt{t}}-\frac{X_{j_1,j_2}-\mu_{j_1,j_2}}{\sqrt{1-t}}\Bigr)\Bigr).
		\]
		Moreover, integration by parts yields
		\[
			\Ex\Bigl(\frac{\partial  f^{\beta,\delta}_k}{\partial x_{j_i,j_2}}(Z_t)(Y_{j_1,j_2}-\mu_{j_1,j_2})\Bigr)=\sum_{(i_1,i_2)\in[n]\times[m]}\sqrt{t}\,\Ex\Bigl(\frac{\partial^2 f^{\beta,\delta}_k}{\partial x_{i_1,i_2}\partial x_{j_1,j_2}}(Z_t)\Bigr)\sigma^Y_{i_1,i_2;j_1,j_2}
		\]
		and
		\[
			\Ex\Bigl(\frac{\partial  f^{\beta,\delta}_k}{\partial x_{j_i,j_2}}(Z_t)(X_{j_1,j_2}-\mu_{j_1,j_2})\Bigr)=\sum_{(i_1,i_2)\in[n]\times[m]}\sqrt{1-t}\,\Ex\Bigl(\frac{\partial^2 f^{\beta,\delta}_k}{\partial x_{i_1,i_2}\partial x_{j_1,j_2}}(Z_t)\Bigr)\sigma^X_{i_1,i_2;j_1,j_2}.
		\]
		Plugging both previous identities into the initial one, we obtain
		\[
		\psi'(t)=\oneover{2}	\sum_{\subalign{(i_1,i_2)&\in[n]\times[m]\\(j_1,j_2)&\in[n]\times[m]}}\Ex\Bigl(\frac{\partial^2 f^{\beta,\delta}_k}{\partial x_{i_1,i_2}\partial x_{j_1,j_2}}(Z_t)\Bigr)(\sigma^Y_{i_1,i_2;j_1,j_2}-\sigma^X_{i_1,i_2;j_1,j_2}).
		\]
		Note that, by construction, 
		\[
		\Ex(f^{\beta,\delta}_k(Y))-\Ex(f^{\beta,\delta}_k(X))=\psi(1)-\psi(0)=\int_0^1	\psi'(t)\,\mathrm{d} t.
		\]
Using \Cref{lem:Gordontrick} in combination with \Cref{lem:signs} shows that under the conditions on the signs of \(\gamma^X_{i_1,i_2;j_1,j_2}- \gamma^Y_{i_1,i_2;j_1,j_2}\) as per assumption, \(\psi'\ge 0\), hence \(\Ex(f^{\beta,\delta}_k(X))\le \Ex(f^{\beta,\delta}_k(Y))\), from which the second claim of the theorem follows by letting \(\beta\to\infty\), thanks to \Cref{lem:unifbound}. Again, \Cref{lem:unifbound} combined this time with \Cref{lem:absbound} and
\[
	\abs[\big]{\Ex(f^{\beta,\delta}_k(X))-\Ex(f^{\beta,\delta}_k(Y))}\le \sup_{t\in[0,1]}\abs{\psi'(t)}\le \frac{\beta k}{4m} (m-k+a_n(2m-1)k\delta)\gamma,
\] shows that
	\[
	\begin{split}
	\abs[\Big]{\Ex\Big(\min_{i_1\in\oneton}\sum_{h=m-k+1}^m X_{i_1,(h)}\Bigr)- &\Ex\Big(\min_{i_1\in\oneton}\sum_{h=m-k+1}^m Y_{i_1,(h)}\Bigr)}\\
	&\le \frac{\beta k}{4m} \bigl(m-k+a_n(2m-1)k\delta\bigr)\gamma+\oneover{\beta\delta}\log n+\oneover{\beta}\log\binom{m}{k} ,
	\end{split}
	\]
that is minimized by 
\[\beta=2\sqrt{\frac{m \log \binom{m}{k}}{(m-k) k\gamma}}
\qquad
\text{and}
\qquad
\delta=\sqrt{\frac{(m-k)\log n}{k a_n(2m-1)\log \binom{m}{k}}},
\] yielding the bounds in the statement.
\qed

\section{Comparison of min-max statistics of two random matrices, one of which is Gaussian}
\label{sec:D12}

\subsection{The language of Malliavin calculus}\label{ss:malliavin}

The reader is referred e.g. to the monograph \cite{NP} for a detailed discussion of the concepts presented in this subsection. 

 Let $\mathfrak H$ be a real separable Hilbert space, and write $\langle \cdot, \cdot \rangle_\mathfrak{H}$ to indicate the corresponding inner product. In what follows, we will write $G=\{G(h) \colon h\in\mathfrak H\}$ to denote an \textit{isonormal Gaussian process} over $\mathfrak H$, that is: $G$ is a (real) centered Gaussian family indexed by $\mathfrak H$ and such that $\Ex \bigl(G(h)G(g)\bigr) = \langle h,g\rangle_\mathfrak{H}$, for all $h,g \in \mathfrak H$. Every $F\in L^2(\sigma(G) )$ admits a \textit{Wiener-It\^o chaos expansion} of the form
\begin{equation}\label{e:chaoss}
F = \Ex(F) +  \sum_{q=1}^\infty I_q(f_q),
\end{equation}
where $f_q$ is an element of the symmetric $q$th tensor product $\mathfrak H^{\odot q}$ (which is uniquely determined by $F$), and  $I_q(f_q)$ is the $q$-th \textit{multiple Wiener-It\^o integral} of $f_q$ with respect to $G$. One writes $F\in \mathbb D^{1,2}$ if 
\begin{align*}
\sum_{q\ge 1} q q!\norm{f_q}_{\mathfrak{H}^{\otimes q}}^2<\infty.
\end{align*}
For $F\in \mathbb D^{1,2}$, we denote by $DF$ the \textit{Malliavin derivative} of $F$. Recall that $DF$ is by definition a random element with values in $\mathfrak H$. The operator $D$ satisfies a crucial \textit{chain rule}: if $\varphi$ is a mapping on $\mathbb{R}^m$ of class $C^1$ with bounded derivatives and if $F_1,\dots,F_m\in\mathbb D^{1,2}$, then $\varphi(F_1,\ldots,F_m)\in\mathbb D^{1,2}$, and also 
\begin{equation}\label{chainrule}
D\varphi(F_1,\ldots,F_m) = \sum_{i=1}^m \partial_i\varphi(F_1,\ldots,F_m)DF_i.
\end{equation}
For general $p>2$, we write $F\in\mathbb D^{1,p}$ if $F\in L^p(\sigma(G))\cap \mathbb D^{1,2} $ and $\Ex\bigl(\norm{DF}_\mathfrak{H}^p\bigr)<\infty$.  The adjoint of $D$, customarily referred to as the \textit{divergence operator} or the \textit{Skorohod integral}, is denoted by $\delta$ and satisfies the duality formula,
\begin{align}\label{duality}
\Ex\bigl(\delta(u)F\bigr) = \Ex\bigl(\langle u, DF\rangle_\mathfrak{H}\bigr)
\end{align} 
for all $F\in\mathbb D^{1,2} $, whenever $u\colon \Omega\to \mathfrak H$ is contained in the domain $\mathrm{ Dom}(\delta)$ of $\delta$. 

\smallskip

The \textit{generator of the Ornstein-Uhlenbeck semigroup}, written $L$, is defined by the relation
$
L F= - \sum_{q\ge 1} q I_q(f_q)
$
for every $F$ as in \eqref{e:chaoss} such that $\sum_{q\ge 1} q^2 q! \norm{f_q}_{\mathfrak H^{\otimes q}}^2<\infty$. The \textit{pseudo-inverse} of $L$, written $L^{-1}$, is the operator defined, as
$
L^{-1} F = - \sum_{q\ge 1} \frac{1}{q} I_q(f_q),
$
for all $F\in L^2(\sigma(G))$ as in \eqref{e:chaoss}. The fundamental relation linking the objects introduced above is the identity
\begin{align}
\label{e:relation}
F= \Ex (F )-\delta(DL^{-1}F),
\end{align}
which is valid for any $F\in L^2(\sigma(G))$ (this relation implies in particular that, for every $F\in L^2(\sigma(G))$, $DL^{-1}F \in \mathrm{ Dom} (\delta)$).

\bigskip
The notation and setting introduced above will prevail for the rest of the section; also, we will systematically assume that the underlying Hilbert space $\mathfrak H$ has infinite dimension.

\subsection{Main estimates}\label{ss:estimates}

We now fix the following objects: \(X=(X_{i_1,i_2})_{(i_1,i_2)\in[n]\times[m]}\) is a centered Gaussian random matrix with covariance matrix \((\sigma_{i_1,i_2;j_1,j_2})_{(i_1,i_2), (j_1,j_2)\in[n]\times[m]}\) (without loss of generality, we can assume that $X$ is extracted from the isonormal Gaussian process $G$);  \(F=(F_{i_1,i_2})_{(i_1,i_2)\in[n]\times[m]}\) is a centered random matrix with entries \(F_{i_1,i_2}\in\mathbb{D}^{1,2}\). We also write \(\sigma_{i_1,i_2}\) as shorthand for \(\sqrt{\Var(X_{i_1,i_2})}\), \(\underline{\sigma}_{i_1}\coloneqq\min_{i_2\in [m] }\sigma_{i_1,i_2}\), \(\underline{\sigma}=\min_{(i_1,i_2)\in[n]\times[m]}\sigma_{i_1,i_2}\) and \(\overline{\sigma}: =\max_{(i_1,i_2)\in[n]\times[m]}\sigma_{i_1,i_2}\).

For simplicity, we will now work with statistics such a the ones appearing on the left-hand side of \eqref{e:mainz} only in the case \(k=1\). To this end, recall that we write \(\min\max X\) to indicate \(\min_{i_1\in\oneton}\max_{i_2\in[m]} X_{i_1,i_2}\) and analogously for \(F\). Our main findings are contained in the statement of the forthcoming Theorem \eqref{thm:main}, providing an upper bound on the Kolmogorov distance between the distributions of the min-max's of \(X\) and \(F\).

In this section, we make the mild assumption that the covariance structure of the random matrix \(X\) is such that 
\begin{equation}
	\label{eq:assumption}
	\Prob\bigl(\abs[\big]{\{(i_1,i_2)\in[n]\times[m]:X_{i_1,i_2}=\min\max X\}}>1\bigr)=0;\tag{A}
\end{equation}
in other words: we require that, with probability one, there exists a unique pair $(i_*, j_*)$ such that $X_{i_*, j_*}= \min\max X$. 
This is, for instance, the case when \(\mathrm{corr}(X_{i_1,i_2},X_{j_1,j_2})<1\) for all distinct pairs \((i_1,i_2)\) and \((j_1,j_2)\),
 or when $X$ is the matrix associated with the order statistics of a vector \(W\), whose components verify \(\mathrm{corr}(W_{i},W_{j})<1\) for all \(i\neq j\), built as described in the proof of \Cref{c:order}. Indeed, by construction, the argument of the min-max of such a matrix is always unique with probability \(1\).

The next statement is the main achievement of the present section. In the special case $n=1$, it generalizes both \cite[Theorem 2]{CCK15} and \cite[Theorem 2.1]{K}. Note that reference \cite{CCK15} only deals with the case in which both $F$ and $X$ are Gaussian. 
\begin{theorem}
	\label{thm:main}
	Let the above assumptions prevail, suppose that \(\underline{\sigma}>0\) and let
\begin{equation}
	\label{eq:Delta}
\Delta\coloneqq\Ex\Biggl(\max_{\substack{(i_1,i_2)\in[n]\times[m]\\(j_1,j_2)\in[n]\times[m]}}\abs[\big]{\scalar{DF_{i_1,i_2},-DL^{-1}F_{j_1,j_2}}-\sigma_{i_1,i_2;j_1,j_2}}\Biggr).
\end{equation}
\begin{enumerate}[\textrm{(a)}]
	\item Let  \(a_{m,i}\coloneqq\Ex\bigl(\max_{i_2\in[m]}\frac{X_{i,i_2}}{\sigma_{i,i_2}}\bigr)\),  \(\alpha_{nm}\coloneqq\oneover{n}\sum_{i=1}^n a_{m,i}\) and \(p_{nm}\coloneqq n/\log{nm}\).
Suppose that there exist constants $\zeta, \zeta'>0$ such that   \(\zeta \leq \underline{\sigma}\leq\overline{\sigma}\leq \zeta'\). Then, there exists a constant \(C>0\), depending only on $\zeta, \zeta'$, such that
\begin{equation*}
\begin{split}
		\sup_{x\in\R}\abs[\big]{\Prob(\min\max F\le x)&-\Prob(\min\max X\le x)}\\
		&\le C \max\bigl(1,\alpha_{nm}^2,\log p_{nm},\log(1/\Delta)\bigr)^{1/3}n^{2/3}(\log nm)^{1/3}\Delta^{1/3}.
\end{split}
\end{equation*}
\item[\textrm{(b)}] Suppose that there exists a constant $\kappa>0$ such that \(\underline{\sigma}\geq \kappa\). Then, there exists a constant \(\tilde C>0\), depending only on $\kappa$, such that
	\begin{equation}\label{e:lage}
\sup_{x\in\R}\abs[\big]{\Prob(\min\max F\le x)-\Prob(\min\max X\le x)}
\le \tilde{C} \,n^{2/3}(\log m)^{1/3}(\log nm)^{1/3}\Delta^{1/3}. 
\end{equation}
\end{enumerate}
\end{theorem}

 \medskip
 
\noindent\textit{Remark.} If $F$ is Gaussian, then the quantity $\Delta$ appearing in \eqref{eq:Delta} is simply the maximal discrepancy -- in absolute value -- between the entries of the covariance matrices. In this special case, our bounds can be compared with \cite[Theorem 2.1]{DHJL}. In particular, specialising such a result to maxima ($r=d$ in the notation of \cite{DHJL}) yields an estimate on the left-hand side of \eqref{e:lage} where the mapping $(n,m)\mapsto n^{2/3}(\log m)^{1/3}(\log nm)^{1/3}$ is replaced by an application  of the type $(n,m)\mapsto n^a m^b$, with $a,b>1$, and $\Delta^{1/3}$ is replaced by a index of discrepancy between the two covariance matrices which is of the order $\Delta$, for $\Delta$ converging to zero. We also observe that --- reasoning as in \cite[Corollary 2.1]{K} --- the bounds in the statement of Theorem \ref{thm:main} continue to hold if the matrices $F$ and $X$ are replaced by those with entries $\abs{F_{i_1,i_2}}$ and $\abs{X_{i_1,i_2}}$, respectively (up to a change in the exact value of the absolute constants $ C, \tilde C$).

 \medskip

We proceed with the proof of \Cref{thm:main}. First, we need a bound on the second derivatives the composition of a smooth function with the approximation function \(f_k^{\beta,\delta}\) defined in Lemma \ref{lem:unifbound} (note that, in the statement below, we consider a generic $k\in [m]$ despite only $k=1$ being relevant for the present section).
\begin{lemma}
	\label{lem:derivativebound}
	Let \(g\colon\R\to\R\) be twice continuously differentiable with bounded first and second derivatives. Then for all \((i_1,i_2),(j_1,j_2)\in[n]\times[m]\),
			\begin{equation*}
	\frac{\partial^2( g\circ f_k^{\beta,\delta})}{\partial x_{i_1,i_2}\partial x_{j_1,j_2}}
	=g''(f_k^{\beta,\delta})\frac{p_{i_1}^{i_2}}{q_{i_1}}\frac{p_{j_1}^{j_2}}{q_{j_1}}+g'(f_k^{\beta,\delta})\beta\Bigl[\delta\frac{p_{i_1}^{i_2}p_{j_1}^{j_2}}{q_{i_1}q_{j_1}}+\frac{\delta_{i_1,j_1}}{q_{i_1}}\bigl(-(1+\delta)p_{i_1}^{i_2}p_{j_1}^{j_2}+ p_{i_1}^{i_2,j_2}\bigr)\Bigr].
	\end{equation*}
	In particular 
	\begin{equation*}
		\sum_{\substack{(i_1,i_2)\in[n]\times[m]\\(j_1,j_2)\in[n]\times[m]}} \abs[\Big]{\frac{\partial^2( g\circ f_k^{\beta,\delta})}{\partial x_{i_1,i_2}\partial x_{j_1,j_2}}}
		\le 
	k^2	\norm{g''}_\infty +2\beta k(1+\delta k)\norm{g'}_\infty.
	\end{equation*}
\end{lemma}
\begin{proof}
		By the chain rule we have 
		\begin{equation*}
			\frac{\partial^2( g\circ f_k^{\beta,\delta})}{\partial x_{i_1,i_2}\partial x_{j_1,j_2}}
			=
			g''(f_k^{\beta,\delta})\frac{\partial f_k^{\beta,\delta}}{\partial x_{i_1,i_2}}\frac{\partial f_k^{\beta,\delta}}{\partial x_{i_1,i_2}}
			+
			g'(f_k^{\beta,\delta})\frac{\partial^2 f_k^{\beta,\delta}}{\partial x_{i_1,i_2}\partial x_{j_1,j_2}},
		\end{equation*}
		which yields the statement because of the computations above. The last inequality follows from the computations of \Cref{lem:absbound}, where the diagonals are also taken into account.
\end{proof}

The next statement generalises Lemma 5 in \cite{CCK15}.

\begin{lemma}
	\label{lem:density}
	Let \(W\) be a Gaussian random matrix in \(\R^{n\times m}\) with \(\Var(W_{i_1,i_2})=1\) for all \((i_1,i_2)\in[n]\times[m]\) and \eqref{eq:assumption} holds. Then the distribution of \(\min\max W\) admits a density with respect to the Lebesgue measure given by
	\begin{equation*}
	g_{n,m}(z)=\phi(z)\sum_{i=1}^n H_i(z)G_i(z),
	\end{equation*}
	where
	\[
		H_i(z)\coloneqq\Prob\Bigl(\max_{\ell\in[m]} W_{i,\ell}=\min\max W\Bigm|  \max_{\ell\in[m]} W_{i,\ell}=z\Bigr)
	\]
	and
	\[
		G_i(z)\coloneqq\sum_{\ell=1}^me^{\Ex(W_{i,\ell})z-\Ex(W_{i,\ell})^2}\Prob\Bigl( W_{i,k}=\max_{\ell\in[m]} W_{i,\ell}\Bigm| W_{i,k}=z\Bigr).
	\]
	Moreover \(z\mapsto e^{\Ex(W_{i,\ell})z-\Ex(W_{i,\ell})^2}\Prob\bigl( W_{i,k}=\max_{\ell\in[m]} W_{i,\ell}\bigm| W_{i,k}=z\bigr)\) is non-decreasing as soon as \(\Ex (W_{i,k})\ge 0\).
	\begin{proof}
Exploiting assumption \eqref{eq:assumption}, one hast that, for every real $t$,  
\[
\Prob\bigl(\min \max W\leq t\bigr) 
= \sum_{i=1}^n \Prob\Big(\max_{\ell\in[m]} W_{i,\ell}=\min\max W\cap   \max_{\ell\in[m]} W_{i,\ell}\leq t\Bigr).
\]
Writing
\[
\begin{split}
\Prob\Big(\max_{\ell\in[m]} W_{i,\ell}&=\min\max W \, \cap \,   \max_{\ell\in[m]} W_{i,\ell}\leq t\Bigr) \\
& = \int_{-\infty}^t \Prob\Big(\max_{\ell\in[m]} W_{i,\ell}=\min\max W \Bigm| \max_{\ell\in[m]} W_{i,\ell} = z\Bigr)\mu_i(\mathrm{d}z),
\end{split}
\]
where $\mu_i$ stands for the law of $\max_{\ell\in[m]} W_{i,\ell}$, we deduce the desired conclusion from Lemmas 5 and 6 in \cite{CCK15}.
\end{proof}

\end{lemma}

The following is a generalization of \cite[Theorem 3]{CCK15}. For every \(\eps>0\) the anti-concentration function of a r.v. \(Y\) is defined as 
\[
	\anti{Y}\coloneqq\sup_{x\in\R}\Prob\bigl(\abs{Y-x}\le\eps\bigr).
\]
If \(Y\) if absolutely continuous with essentially bounded density \(f\) then it follows from the definition that
\[
	\anti{Y}\le 2\eps\, \norm{f}_\infty,
\]
where \(\norm{\cdot}_\infty\) is the essential supremum.
\begin{lemma}[Anti-concentration inequality, first variant]
	\label{prop:anticoncentration}
	
	 There exists \(C>0\) that depends only on \(\underline{\sigma}\) and \(\overline{\sigma}\) such that for all \(\eps>0\) 
	\begin{equation}
		\anti{\min\max X}
		\le C\eps\Bigl(\sum_{i=1}^n a_{m,i}+n\max\bigl(1,\sqrt{\log(\underline{\sigma}/\eps)}\bigr) \Bigr).
	\end{equation}
\end{lemma}
\begin{proof} We divide the proof in two steps.
	\begin{itemize}
		\item[(i)] Reduction to unit variance. Let \(x\ge 0\) arbitrary and let
		\[
			W_{i_1,i_2}\coloneqq\frac{X_{i_1,i_2}-x}{\sigma_{i_1,i_2}}+\frac{x}{\underline{\sigma}}.
		\]
		Then \(\mu_{i_1,i_2}\coloneqq\Ex(W_{i_1,i_2})=x\bigl(\oneover{\underline{\sigma}}-\oneover{\sigma_{i_1,i_2}}\bigr)\ge 0\) and \(\Var(W_{i_1,i_2})=1\). Let \(Z\coloneqq\min\max W\). Since the function \(\min\max\) is non-decreasing in each argument, we have
		\begin{equation}
		\begin{split}
		\Prob\bigl(\abs{\min\max X-x}\le\eps\bigr)
		&\le
		\Prob\Bigl(\abs[\Big]{\min_{i_1}\max_{i_2}\frac{X_{i_1,i_2}-x}{\sigma_{i_1,i_2}}}\le\frac{\eps}{\underline{\sigma}}\Bigr)\\
		&\le
		\sup_{y\in\R}\Prob\Bigl(\abs[\Big]{\min_{i_1}\max_{i_2}\frac{X_{i_1,i_2}-x}{\sigma_{i_1,i_2}}+\frac{x}{\underline{\sigma}}-y}\le\frac{\eps}{\underline{\sigma}}\Bigr)\\
		&=
		\sup_{y\in\R}\Prob\Bigl(\abs{Z-y}\le\frac{\eps}{\underline{\sigma}}\Bigr).
		\end{split}
		\end{equation}
		\item[(ii)] We proceed with bounding the density of \(Z\). Since \(W_{i_1,i_2}\sim\mathcal{N}(\mu_{i_1,i_2},1)\), by \Cref{lem:density}, assuming \eqref{eq:assumption} we have that the density of \(Z\) has the form
		\begin{equation}
			g_{n,m}(z)=\phi(z)\sum_{i=1}^n H_i(z)G_i(z)\le \phi(z)\sum_{i=1}^nG_i(z)
		\end{equation}
		\end{itemize}
	We know from \cite[Lemma 7]{CCK15} that 
	\[
		\phi(z)G_i(z)\le 2\max(z,1)\exp\Bigl(-\frac{\max(z-\overline{z}-a_{m,i},0)^2}{2}\Bigr)\le 2(\overline{z}+a_{m,i}+1),
	\]
	where \(\overline{z}\coloneqq x\bigl(\oneover{\underline{\sigma}}-\oneover{\overline\sigma}\bigr)\),  hence
	\[
		g_{n,m}(z)\le 2\sum_{i=1}^n(\overline{z}+a_{m,i}+1).
	\]
	In particular, for all \(y\in\R\) and \(t>0\) we have
	\[
		\Prob\Bigl(\abs{Z-y}\le\frac{\eps}{\underline{\sigma}}\Bigr)\le 4\frac{\eps}{\underline{\sigma}}\sum_{i=1}^n(\overline{z}+a_{m,i}+1)
	\]
	and using step (i) we get
	\[
		\Prob\bigl(\abs{\min\max X-x}\le\eps\bigr)\le	\sup_{y\in\R}\Prob\Bigl(\abs{Z-y}\le\frac{\eps}{\underline{\sigma}}\Bigr)\le 4\frac{\eps}{\underline{\sigma}}\sum_{i=1}^n(\overline{z}+a_{m,i}+1)
	\]
	Repeating the argument with \(x<0\) one gets instead, one gets
		\[
	\Prob\bigl(\abs{\min\max X-x}\le\eps\bigr)\le4\frac{\eps}{\underline{\sigma}}\Bigl(n\abs{x}\Bigl(\oneover{\underline{\sigma}}-\oneover{\overline\sigma}\Bigr)+\sum_{i=1}^n 1+a_{m,i}\Bigr).
	\]
	If \(\underline{\sigma}=\overline{\sigma}=\sigma\) then \(	\Prob\bigl(\abs{\min\max X-x}\le\eps\bigr)\le \frac{4\eps}{\sigma}\sum_{i=1}^n 1+a_{m,i}\). On the contrary if \(\underline{\sigma}<\overline{\sigma}\), note that the claim is true trivially for \(\underline{\sigma}/\eps<1\). If \(\eps\le \underline{\sigma}\) instead, for \(\abs{x}\ge \eps+\overline{\sigma}\bigl(\oneover{n}\sum_{i=1}^n a_{m,i}+\sqrt{2\log(\underline{\sigma}/\eps)}\bigr)\), since the min-max is a Lipschitz function, we can use the Gaussian deviation inequality (see \cite[Theorem 7.1]{L})
	\[
	\begin{split}
		\Prob\bigl(\abs{\min\max X-x}\le\eps\bigr)
		&\le
		\Prob\bigl(\min\max X\ge\abs{x}-\eps\bigr) \\
		& \le
			\Prob\Bigl(\min\max X\ge\frac{\overline{\sigma}}{n}\sum_{i=1}^n a_{m,i}+\overline{\sigma}\sqrt{2\log(\underline{\sigma}/\eps)}\Bigr)\\
		&\le
			\Prob\Bigl(\min\max X\ge\Ex(\min\max X)+\overline{\sigma}\sqrt{2\log(\underline{\sigma}/\eps)}\Bigr)
			\le 
			\frac{\eps}{\underline{\sigma}},
	\end{split}
	\]
	where we used the fact that \( \frac{\overline{\sigma}}{n}\sum_{i=1}^n a_{m,i}\ge\oneover{n}\sum_{i=1}^n\Ex\bigl(\max_{i_2\in[m]}X_{i,i_2}\bigr)\ge \Ex(\min \max X)\), since \(\min\max X\le \max_{i_2\in[m]}X_{i,i_2}\) for all \(i\in[n]\). When \(\abs{x}\le \eps+\overline{\sigma}\bigl(\oneover{n}\sum_{i=1}^n a_{m,i}+\sqrt{2\log(\underline{\sigma}/\eps)}\bigr)\le\abs{x}\le \underline{\sigma}+\overline{\sigma}\bigl(\oneover{n}\sum_{i=1}^n a_{m,i}+\sqrt{2\log(\underline{\sigma}/\eps)}\bigr) \) instead, we get
	\begin{equation*}
	\begin{split}
		\Prob\bigl(\abs{\min\max X-x}\le\eps\bigr)
		&\le 4\frac{\eps}{\underline{\sigma}}\Bigl(n\abs{x}\Bigl(\oneover{\underline{\sigma}}-\oneover{\overline\sigma}\Bigr)+\sum_{i=1}^n 1+a_{m,i}\Bigr)\\
		&\le
		4\frac{\eps}{\underline{\sigma}}\Bigl(n\Bigl(1-\frac{\underline{\sigma}}{\overline{\sigma}}\Bigr)+n\Bigl(\frac{\overline{\sigma}}{\underline{\sigma}}-1\Bigr)\sqrt{2\log(\underline{\sigma}/\eps)}+\Bigl(\frac{\overline{\sigma}}{\underline{\sigma}}-1\Bigr)\sum_{i=1}^n a_{m,i}\Bigr)\\
		&\le C\eps\Bigl(\sum_{i=1}^n a_{m,i}+n\max\bigl(1,\sqrt{\log(\underline{\sigma}/\eps)}\bigr) \Bigr),
	\end{split}
	\end{equation*}
which concludes the proof.
\end{proof}

\noindent\textit{Remark.} Rewriting the bound of \Cref{prop:anticoncentration} as 
\[
	\anti{\min\max X}\le Cn\eps(\alpha_{nm}+\max\bigl(1,\sqrt{\log(\underline{\sigma}/\epsilon)}\bigr),
\]
we see that the multiplicative factor \(n\) cannot be improved. In fact, suppose that the matrix \(X\) is composed by \(n\) i.i.d. rows, which are copies of a \(m\text{-dimensional}\) standard Gaussian row vector \(X'\), and let \(\alpha\coloneqq \Ex(\max X')\). 
Then
\[
\Prob(\min\max X\ge z)=\Prob(\max X'\ge z)^n,
\]
and 
\[
	\frac{\mathrm{d}}{\mathrm{d} z}\Prob(\min\max X\ge z)=n \Prob(\max X'\ge z)^{n-1}\frac{\mathrm{d}}{\mathrm{d} z}\Prob(\max X'\ge z).
\]
Exploiting the sub-Gaussian deviation inequality for \(\max X'\) we get that the last expression is less or equal than 
\[
n \exp\Bigl(-\frac{(n-1)\max(z-\alpha,0)^2}{2}\Bigr)\times 2\max(z,1)\exp\Bigl(-\frac{\max(z-\alpha,0)^2}{2}\Bigr),
\]
which is uniformly bounded from above by \(2n(\alpha+1) \), the desired order.

\bigskip

The following is a generalization of \cite[Lemma 4.4]{CCK16}.

\begin{lemma}[Anti-concentration inequality, second variant]
	\label{lem:anti-con2}
	For all \(\eps>0\)
	\[
	\anti{\min\max X}\le2\sqrt{2}\eps\sum_{i_1=1}^n\oneover{\underline{\sigma}_{i_1}}\bigl(\sqrt{2}+\sqrt{\log m}\bigr)\le 2\sqrt{2}\frac{\eps}{\underline{\sigma}}n\bigl(\sqrt{2}+\sqrt{\log m}\bigr).
	\]
\end{lemma}
\begin{proof}
	Let \(\Sigma_{i_1}\) be the covariance matrix of the row vector \(X_{i_1,\cdot}\), in particular \(X_{i_1,\cdot}\stackrel{\text{d}}{=}\Sigma_{i_1}^{1/2} Z_{i_1}+\mu_{i_1}\) for some \(Z_{i_1}\sim\mathcal{N}(0,\mathrm{I}_m)\). Note that the \(i_2\text{-th}\) row of \(\Sigma_{i_1}^{1/2}\) can be written as \(\sigma_{i_1,i_2}v_{i_1,i_2}\) for some unit-norm row vector \(v_{i_1,i_2}\in\R^{1\times m}\), which yields
	\[
	B_x\coloneqq\Bigl\{\min_{i_1\in\oneton}\max_{i_2\in[m]} (\Sigma_{i_1}^{1/2} Z_{i_1}+\mu_{i_1})_{i_2}\le x\Bigr\}=\bigcup_{i_1\in\oneton}\Bigl\{\forall i_2\in[m]\; v_{i_1,i_2}Z_{i_1}\le \frac{x-\mu_{i_1,i_2}}{\sigma_{i_1,i_2}}\Bigr\} \eqqcolon\bigcup_{i_1\in[n]}C_{i_1,x}.
	\]
	In particular, since \(\min\max X\) is absolutely continuous, its density \(f\) is given by
	\[
	f(x)=\lim_{\eps\to 0} \oneover{\eps}\Prob(B_{x+\eps}\setminus B_x),
	\]
	for almost all \(x\in\R\). For all \(i_1\in[n]\) we have
	\[
	C_{i_1,x+\eps}=\Bigl\{\forall i_2\in[m]\; v_{i_1,i_2}Z_{i_1}\le \frac{x+\eps-\mu_{i_1,i_2}}{\sigma_{i_1,i_2}}\Bigr\}
	\!\subseteq\!
	\Bigl\{\forall i_2\in[m]\; v_{i_1,i_2}Z_{i_1}\le \frac{x-\mu_{i_1,i_2}}{\sigma_{i_1,i_2}}+\frac{\eps}{\underline{\sigma}_{i_1}}\Bigr\}
	=C_{i_1,x}^{\eps/\underline{\sigma}_{i_1}},
	\]
	hence
	\[
	B_{x+\eps}\setminus B_x
	=\bigcup_{i_1\in[n]}C_{i_1,x+\eps}\setminus \bigcup_{i_1\in\oneton} C_{i_1,x}
	\subseteq \bigcup_{i_1\in[n]}C_{i_1,x}^{\eps/\underline{\sigma}_{i_1}}\setminus \bigcup_{i_1\in\oneton} C_{i_1,x}
	\subseteq \bigcup_{i_1\in[n]} C_{i_1,x}^{\eps/\underline{\sigma}_{i_1}}\setminus C_{i_1,x}.
	\]
	By the union bound, we deduce that 
	\[
	f(x)\le \lim_{\eps\to 0} \sum_{i_1=1}^n\oneover{\eps}\Prob\bigl(C_{i_1,x}^{\eps/\underline{\sigma}_{i_1}}\setminus C_{i_1,x}\bigl).
	\]
	Using Nazarov's inequality (see \cite{N}) on each term of the last sum gives that
	\[
		\lim_{\eps\to 0}\oneover{\eps}\Prob\bigl(C_{i_1,x}^{\eps/\underline{\sigma}_{i_1}}\setminus C_{i_1,x}\bigl)\le \frac{2\sqrt{2}}{\underline{\sigma}_{i_1}}\bigl(\sqrt{2}+\sqrt{\log m}\bigr),
	\]
	which allows us to conclude.
\end{proof}

\medskip

\noindent\textit{Remark.} In the previous Lemma, it is not necessary for \(X\) to be centered.
\medskip

In the case $n=1$,  the next statement is a generalization of \cite[Theorem 1]{CCK15} (case of $F$ Gaussian) and \cite[Theorem 2.1]{K} (general case).

\begin{proposition}
	\label{prop:smoothbound}
		Let \(g\colon\R\to\R\) be twice continuously differentiable with bounded first and second derivatives. Then
		\begin{equation*}
		\abs[\big]{\Ex\bigl(g\circ f_1(F)\bigr)- \Ex\bigl(g\circ f_1(X)\bigr)}
		\le 
		\Bigl(	\frac{1}{2}\norm{g''}_\infty +\beta (1+\delta )\norm{g'}_\infty\Bigr) \Delta
		\end{equation*}
In particular, in view of \Cref{lem:unifbound}, it also holds
		\begin{equation*}
			\abs[\big]{\Ex\bigl(g(\min\max F)\bigr)- \Ex\bigl(g (\min\max X)\bigr)}
			\le
				\Bigl(	\frac{1}{2}\norm{g''}_\infty +\beta (1+\delta )\norm{g'}_\infty\Bigr) \Delta+ \Bigl( \oneover{\beta}\log m +\oneover{\beta\delta}\log n\Bigr)\norm{g'}_\infty.
		\end{equation*}
\end{proposition}
	\begin{proof}
		We may assume that \(F\) and \(X\) are independent, without loss of generality. Consider their interpolation given by \begin{equation*}
		Z(t)\coloneqq \sqrt{t}F+\sqrt{1-t}F,
		\end{equation*}
		for all \(t\in[0,1]\). 
		Let \(\phi\coloneqq g\circ f_1\) (note that by \Cref{lem:derivativebound}, all second partial derivatives of $\varphi$ bounded) and consider the function \(\Psi(t)\coloneqq\Ex\bigl(\phi(Z(t)) \bigr)\). Then \(\Psi\) is differentiable in \([0,1]\) and there its derivative is equal to
		\begin{equation*}
			\Psi'(t)=\oneover{2}\sum_{(j_1,j_2)\in\oneton\times\onetom}\Ex \Bigl( \frac{\partial\phi}{\partial x_{j_1,j_2}}(Z(t)) \Bigl(\frac{F_{j_1,j_2}}{\sqrt{t}}- \frac{X_{j_1,j_2}}{\sqrt{1-t}}\Bigr)\Bigr).
		\end{equation*}
		By independence and integration by parts we deduce that
		\begin{equation*}
			\sum_{(j_1,j_2)\in\oneton\times\onetom}\Ex \Bigl( \frac{\partial\phi}{\partial x_{j_1,j_2}}(Z(t))  \frac{X_{j_1,j_2}}{\sqrt{1-t}}\Bigr)
			=
			\sum_{\substack{(i_1,i_2)\in[n]\times[m]\\(j_1,j_2)\in[n]\times[m]}}\Ex \Bigl( \frac{\partial^2\phi}{\partial x_{i_1,i_2}\partial x_{j_1,j_2}}(Z(t)) \sigma_{i_1,i_2;j_1,j_2}\Bigr).
		\end{equation*}
		Analogously, reasoning as in the proof of \cite[Theorem 6.1.1]{NP} yields that
			\begin{equation*}
		\sum_{(j_1,j_2)\in\oneton\times\onetom}\Ex \Bigl( \frac{\partial\phi}{\partial x_{j_1,j_2}}(Z(t))  \frac{F_{j_1,j_2}}{\sqrt{t}}\Bigr)
		=
		\!\!\!\!\sum_{\substack{(i_1,i_2)\in[n]\times[m]\\(j_1,j_2)\in[n]\times[m]}}\Ex \Bigl(  \frac{\partial^2\phi}{\partial x_{i_1,i_2}\partial x_{j_1,j_2}}(Z(t))  \scalar{DF_{i_1,i_2},-DL^{-1}F_{j_1,j_2}}\Bigr).
		\end{equation*}
	Hence
		\begin{equation*}
	\Psi'(t)=\oneover{2}	\sum_{\substack{(i_1,i_2)\in[n]\times[m]\\(j_1,j_2)\in[n]\times[m]}}\Ex \Bigl(  \frac{\partial^2\phi}{\partial x_{i_1,i_2}\partial x_{j_1,j_2}}(Z(t))  \bigl(\scalar{DF_{i_1,i_2},-DL^{-1}F_{j_1,j_2}}-\sigma_{i_1,i_2;j_1,j_2} \bigr)\Bigr).
	\end{equation*}
	We eventually conclude that 
		\begin{equation*}
			\abs[\big]{\Ex\bigl(\phi( F)\bigr)- \Ex\bigl(\phi(X)\bigr)}
			\le 
			\norm{\phi'}_{\infty}
			\le
			\Bigl(	\frac{1}{2}\norm{g''}_\infty +\beta (1+\delta )\norm{g'}_\infty\Bigr) \Delta,
		\end{equation*}
	using \Cref{lem:derivativebound} with \(k=1\).
	\end{proof}
\begin{lemma}
	\label{lem:5eps}
	There exists a constant \(C>0\) such that for all \(A\subset\R\) measurable we have
	\begin{equation*}
	\Prob(\min\max F\in A)\le \Prob(\min\max X\in A^{5\eps})+C\,\frac{\log nm}{\eps^2}\,\Delta,
	\end{equation*}
	for all \(\eps>0\).
\end{lemma}
\begin{proof}
	Let \(\delta=1\) and \(\eps=\oneover{\beta}\log nm\). Then by \Cref{lem:unifbound} we get
	\begin{equation*}
	\Prob(\min\max F\in A)
	\le 
	\Prob(f_1^{\beta,1}(F)\in A^\eps)=\Ex\bigl(\indic{A^{\eps}}(f_1^{\beta,1}(F)) \bigr).
	\end{equation*}
Now we use the fact that there exists a universal constant \(c>0\) such that for all \(\eps>0\) there exists a smooth function \(g\colon\R\to\R\) such that \(\norm{g'}_\infty\le \oneover{\eps}\),  \(\norm{g''}_\infty\le \frac{c}{\eps}\), and \(\indic{A^\eps}\le g\le \indic{A^{4\eps}}\). For such \(g\), by monotonicity and inequality above we get
\begin{equation*}
	\Prob(\min\max F\in A)\le\Ex\bigl(\indic{A^{\eps}}(f_1^{\beta,\delta}(F) \bigr)\le \Ex\bigl(g(f_1^{\beta,\delta}(F)) \bigr).
\end{equation*}
Now by \Cref{prop:smoothbound} we get
\begin{equation*}
	\abs[\big]{\Ex\bigl(g(f^{\beta,1}_1( F))\bigr)- \Ex\bigl(g (f^{\beta,1}_1(X))\bigr)}
\le 
\Bigl(	\frac{c}{2\eps^2} +\frac{2c}{\eps^2}\log nm\Bigr) \Delta= \frac{5c}{2}\,\frac{\log nm}{\eps^2}\,\Delta.
\end{equation*}
Since
\begin{equation*}
	\Ex\bigl(g (f^{\beta,1}_1(X))\bigr)\le \Ex\bigl(\indic{A^{4\eps}}(f_1^{\beta,1}(X) \bigr)\le \Ex\bigl(\indic{A^{5\eps}}(f_1^{\beta,1}(X)) \bigr)=\Prob(\min\max X\in A^{5\eps}),
\end{equation*}
we can conclude.
\end{proof}
	Finally, we need one more technical Lemma, proved in \cite[Lemma A.3]{K}.
\begin{lemma}
	\label{lem:tech}
 Consider two random variables \(U\) and \(V\). Suppose that there exist two positive constants \(\eps_1\) and \(\eps_2\) such that for all measurable set \(A\subseteq\R\) \begin{equation*}
 \Prob(U\in A)\le \Prob(V\in A^{\eps_1})+\eps_2.
 \end{equation*}
 Then
 \begin{equation*}
 \sup_{x\in\R}\abs[\big]{\Prob(U\le x)-\Prob(V\le x)}\le \sup_{x\in\R}\Prob(\abs{V-x}\le\eps_1)+\eps_2.
 \end{equation*}
\end{lemma}
We are now ready to prove our main results.
\begin{proof}[Proof of \Cref{thm:main} \textrm{(a)}]
	Note that if \(\Delta\ge 1\) the result is trivially true. So we can assume \(\Delta\in(0,1)\).
	
	 \Cref{lem:5eps} allows to use \Cref{lem:tech} with \(\eps_1=5\eps\) and \(\eps_2=C\,\frac{\log nm}{\eps^2}\,\Delta\), to get
	\begin{equation}
		\label{eq:triang}
		\sup_{x\in\R}\abs[\big]{\Prob(\min\max F\le x)-\Prob(\min\max X\le x)}
		\le 
			\sup_{x\in\R}\Prob(\abs{\min\max X-x}\le 5\eps)+C\,\frac{\log nm}{\eps^2}\,\Delta.
	\end{equation}
	We can now use \Cref{prop:anticoncentration} to estimate the first summand, hence
	\begin{equation*}\begin{split}
		\sup_{x\in\R}\abs[\big]{\Prob(\min\max F\le x)&-\Prob(\min\max X\le x)}\\
	&\le C'\eps\Bigl(\sum_{i=1}^n a_{m,i}+n\max\bigl(1,\sqrt{\log(\underline{\sigma}/\eps)}\bigr) \Bigr)+C\,\frac{\log nm}{\eps^2}\,\Delta\\
	&\le C''\eps n\max\bigl(1,\alpha_{nm},\sqrt{\log(1/\eps)}\bigr)+C\,\frac{\log nm}{\eps^2}\,\Delta.
	\end{split}
	\end{equation*}
	Let \(p_{nm}=n/ \log nm\).
	We can estimate the right hand side by choosing
	\begin{equation*}
			\eps^3=\frac{\Delta}{p_{nm}\max\bigl(1,\alpha_{nm},\sqrt{\log p_{nm}},\sqrt{\log(1/\Delta)}\bigr)},
	\end{equation*}
	which yields
	\begin{equation*}
	\begin{split}
		\sup_{x\in\R}\abs[\big]{\Prob(\min\max F\le x)&-\Prob(\min\max X\le x)}\\
		\le  &n^{2/3}\log(nm)^{1/3}\Delta^{1/3}\Bigl[C''\frac{\max\bigl(1,\alpha_{nm},\sqrt{\log(1/\eps)}\bigr)}{\max\bigl(1,\alpha_{nm},\sqrt{\log p_{nm}},\sqrt{\log(1/\Delta)}\bigr)^{1/3}}\\
		&\hspace{100pt}+C\max\bigl(1,\alpha_{nm},\sqrt{\log p_{nm}},\sqrt{\log(1/\Delta)}\bigr)^{2/3}\Bigr].
		\end{split}
	\end{equation*}
	Since \(\max\bigl(1,\alpha_{nm},\log p_{nm},\sqrt{\log(1/\Delta)}\bigr)^{2/3}=\max\bigl(1,\alpha_{nm}^2,(\log p_{nm})^2,\log(1/\Delta)\bigr)^{1/3}\) we are  done if we show that  there exists a constant \(\tilde C\) such that
	\[
	\max\bigl(1,\alpha_{nm}^2,\log(1/\eps)\bigr)\le \tilde C \max\bigl(1,\alpha_{nm}^2,\log p_{nm},\log(1/\Delta)\bigr).
	\]
	Let \(\xi\coloneqq \max\bigl(\alpha_{nm},\sqrt{\log p_{nm}},\sqrt{\log(1/\Delta})\bigr)\). Then
	\[
	\log\Bigl(\frac{1}{\eps}\Bigr)=\frac{1}{3}\log\frac{p_{nm} \max(1,\log p_{nm},\xi)}{\Delta}\le\max\Bigl(\log p_{nm},\log(1/\Delta),\max(0,\log\log p_{nm},\log\xi)\Bigr).
	\]
	Hence,
	\[
	\log\Bigl(\frac{1}{\eps}\Bigr)\le \max(1,\log \alpha_{nm},\log p_{nm},\log (1/\Delta),
	\]
	which concludes the first claim.
	\end{proof}
	\begin{proof}[Proof of \Cref{thm:main} \textrm{(b)}]
		As in \cref{eq:triang} we have	
		\[
		\sup_{x\in\R}\abs[\big]{\Prob(\min\max F\le x)-\Prob(\min\max X\le x)}
		\le 
		\sup_{x\in\R}\Prob(\abs{\min\max X-x}\le 5\eps)+C\,\frac{\log nm}{\eps^2}\,\Delta.
		\]
		We can apply \Cref{lem:anti-con2} to obtain
		\[
		\sup_{x\in\R}\abs[\big]{\Prob(\min\max F\le x)-\Prob(\min\max X\le x)}
		\le C' \eps n\sqrt{\log m}+C\,\frac{\log nm}{\eps^2}\,\Delta.
		\]
		for some \(C'>0\) which depends only on \(m\). The last expression is minimized by choosing
		\[
		\eps=\Bigl(\frac{2C(\log nm)\Delta}{C'n\sqrt{\log m}}\Bigr)^{1/3},
		\]
		which yields
		\[
		\sup_{x\in\R}\abs[\big]{\Prob(\min\max F\le x)-\Prob(\min\max X\le x)}
		\le \tilde{C}\, n^{2/3}(\log m)^{1/3}(\log nm)^{1/3}\Delta^{1/3},
		\]
		concluding the proof.
	\end{proof}

\section{Application to matrices of multiple stochastic integrals}\label{s:mwi}

We will now apply our previous findings to matrices of multiple Wiener-It\^o integrals, as introduced in Section \ref{ss:malliavin} (whose setting will prevail throughout).

\subsection{A general estimate}

%

Let \(q,N\in\N\) and consider three sequences of natural numbers \(d=d(N)\), \(n=n(N)\) and \(m=m(N)\).  For every \((i_1,i_2)\in[n]\times[m]\), we consider a random variable of the type
\[
 F_{i_1, i_2} = F^N_{i_1,i_2}\coloneqq I_q(f^N_{i_1, i_2}),
\]
where $I_q$ indicates a multiple stochastic integral of order $q\geq 2$ and $f_{i_1,i_2} = f^N_{i_1,i_2} \in {\mathfrak H}^{\odot q}$ (when there is no risk of confusion, and in order to simplify the presentation, we will sometimes avoid to write the superscript $N$).

\begin{proposition}
\label{cor:polynomial}
	Suppose that for all \(N\in\N\), \(X^N=(X^N_{i_1,i_2})_{(i_1,i_2)\in[n]\times[m]}\) is a centered Gaussian random matrix with covariance matrix \((\sigma^N_{i_1,i_2;j_1,j_2})_{(i_1,i_2)\in[n]\times[m]}\) and \(F^N=(F^N_{i_1,i_2})_{(i_1,i_2)\in[n]\times[m]}\) is the random matrix described as above. Suppose moreover that there exists a constant \(c>0\) such that \(\underline\sigma^N\ge c\) for all \(N\in\N\) (where we used the same notation introduced at the beginning of Section \textrm{\ref{ss:estimates}}). If
	\[
	A\coloneqq\max_{\substack{(i_1,i_2)\in[n]\times[m]\\(j_1,j_2)\in[n]\times[m]}}\abs[\big]{\sigma^N_{i_1,i_2;j_1,j_2}-\Ex[F_{i_1,i_2}F_{j_1,j_2}]}n^2(\log m)(\log nm)
	\]
	and
	\[
	B\coloneqq\max_{(i_1,i_2)\in[n]\times[m]}\bigl(\Ex( F^4_{i_1,i_2})-3\Ex(F_{i_1,i_2}^2)^2\bigr) n^4(\log m)^2(\log nm)^{2q},
	\]
	then there exists a constant \(C>0\) independent of \(N\) such that
		\[
	\sup_{x\in\R}\abs[\big]{\Prob(\min\max F^N\le x)-\Prob(\min\max X^N\le x)}\le C(A^{1/3}+B^{1/6}).
	\]
\end{proposition}

\medskip

\noindent\textit{Remark}.
The content of Proposition \ref{cor:polynomial} can be regarded as further confirmation of the so-called \textit{(multidimensional) fourth moment phenomenon} (see e.g. \cite[Chapters 5 and 6]{NP}). According to this notion, if $\{F_n\}$ is a sequence of random vectors whose components belong to Gaussian Wiener chaoses of fixed orders, then $\{F_n\}$ verifies a multidimensional central limit theorem if and only if the covariance matrices of  $\{F_n\}$ converges pointwise, and the fourth cumulants of its components converge to zero. The main contribution of Proposition \ref{cor:polynomial} is that of providing (in the case of min-max statistics) a bound with explicit dimensional dependences. See {\cite{web}} for a constantly updated repository of papers connected to fourth moment theorems and related results. 

\medskip

\noindent\textit{Remark.}	Note that for \(n=1\) and \(q=2\) we recover \cite[Theorem 3.1]{K}.

\medskip

\begin{proof}

We know from \cite[Lemma 2.2]{K} that, for \(\Delta \) as defined in \eqref{eq:Delta}, one has that
\[
\begin{split}
\Delta\le \max_{\substack{(i_1,i_2)\in[n]\times[m]\\(j_1,j_2)\in[n]\times[m]}}\abs[\big]{&\sigma^N_{i_1,i_2;j_1,j_2}-\Ex[F_{i_1,i_2}F_{j_1,j_2}]}\\&+C_q\log^{q-1}(2n^2m^2-1+e^{q-2})\!\!\!\!\!\max_{(i_1,i_2)\in[n]\times[m]}\sqrt{\Ex( F^4_{i_1,i_2})-3\Ex(F_{i_1,i_2}^2)^2},
\end{split}
\]
for some constant \(C_q\) depending only on \(q\). Note that since \(q\) is fixed we can bound \(\log^{q-1}(2n^2m^2-1+e^{q-2})\le \tilde c\, \log^{q-1}(nm)\) for some constant \(\tilde c>0\). The conclusion is reached by applying \Cref{thm:main} (b) as follows
\[
\begin{split}
	\sup_{x\in\R}\abs[\big]{\Prob(\min\max F^N\le x)&-\Prob(\min\max X^N\le x)}\\
	&\le C_1 \Bigl(\max_{\substack{(i_1,i_2)\in[n]\times[m]\\(j_1,j_2)\in[n]\times[m]}}\abs[\big]{\sigma^N_{i_1,i_2;j_1,j_2}-\Ex[F_{i_1,i_2}F_{j_1,j_2}]}n^2(\log m)(\log nm)\Bigr)^{1/3}\\
	&+ C_2\Bigl(\max_{(i_1,i_2)\in[n]\times[m]}\bigl(\Ex( F^4_{i_1,i_2})-3\Ex(F_{i_1,i_2}^2)^2\bigr) n^4(\log m)^2(\log nm)^{2q}\Bigr)^{1/6},
\end{split}
\]
where \(C_1, C_2>0\) are constants that do not depend on \(N\), thanks to the fact that \(\underline\sigma^N\) is bounded from below by an absolute constant. 
\end{proof}

\medskip

\medskip

\subsection{An illustration}
\label{sec:stat}
We will now briefly illustrate our findings with an example inspired by the statistical procedures for testing the absence of lead-lag effects in time series, as put forward in \cite[Section 4.1]{K}. See the remark at the end of this section for a statistical interpretation of our findings.  

\medskip

We start by considering a 4-dimensional Gaussian process \(Z(t)=(B_1(t),B_2(t), \tilde B_1(t),\tilde B_2(t))\) on the real line, with the following characteristics: (a) each coordinate of $Z$ is a standard Brownian motion issued from zero, (b) $B_i$ and $\tilde B_i$ are independent, for $i=1,2$, (c) the dependence among other pairs of coordinates of $Z$ is arbitrary. We also write $B \coloneqq (B_1, B_2)$ and $\tilde B \coloneqq (\tilde B_1, \tilde B_2)$, and denote by \(W^\theta(\cdot)=\tilde B(\cdot-\theta)\) the process $\tilde B$ translated in the direction \(\theta\in\R\). 


\smallskip

For some \(T,b,w>0\), we now suppose to observe the process \(B\), respectively \(W^\theta\), at a finite set of points in time \(\mathcal{T}_B=\{0,\frac{T}{bN},\frac{2T}{b N},...,\frac{\lfloor b N\rfloor T}{b N}\}\),  respectively \(\mathcal{T}_{W^\theta}=\{0,\frac{T}{wN},\frac{2T}{wN},...,\frac{\lfloor w N\rfloor T}{w N}\}\), in such a way that  \(\card{\mathcal{T}_B}\sim b N\) and \(\card{\mathcal{T}_W}\sim w N\). For every $\theta \in \R$ , we also introduce the following two (centered) statistics $U_1^N(\theta), \, U_2^N(\theta)$, that can be seen as special cases of the general class defined in \cite[Introduction and Section 4.1]{K},
\begin{align*}
	U_1^N(\theta)
	&=\sum_{i=1}^{\lfloor b N\rfloor} \sum_{j=1}^{\lfloor w N\rfloor} \Bigl(B_1\Bigl(\frac{iT}{b N}\Bigr)-B_1\Bigl(\frac{(i-1)T}{b N}\Bigr)\Bigr)\\
	&\qquad\qquad\qquad\times\Bigl(W^\theta_1\Bigl(\frac{jT}{w N}\Bigr)-W^\theta_1\Bigl(\frac{(j-1)T}{w N}\Bigr)\Bigr)\indic{\{(\frac{(i-1)T}{b N},\frac{iT}{b N}]\cap(\frac{(j-1)T}{w N},\frac{jT}{w N}]\neq\emptyset\}}
\end{align*}
and
\begin{align*}
	U_2^N(\theta)
	&=\sum_{i=1}^{\lfloor b N\rfloor} \sum_{j=1}^{\lfloor w N\rfloor} \Bigl(B_2\Bigl(\frac{iT}{b N}\Bigr)-B_2\Bigl(\frac{(i-1)T}{b N}\Bigr)\Bigr)\\
	&\qquad\qquad\qquad\times\Bigl(W^\theta_2\Bigl(\frac{jT}{w N}\Bigr)-W^\theta_2\Bigl(\frac{(j-1)T}{w N}\Bigr)\Bigr)\indic{\{(\frac{(i-1)T}{b N},\frac{iT}{b N}]\cap(\frac{(j-1)T}{w N},\frac{jT}{w N}]\neq\emptyset\}}.
\end{align*}
We are interested in the fluctuations of the following statistic
\[
	\sqrt{N}\min_{i\in\{1,2\}}\max_{\theta\in\Theta_N}\abs{U^N_i(\theta)},
\]
where \(\Theta_N\) is an index set such that \(\card{\Theta_N}=m(N)\in\N\). In order to study the asymptotic properties of the aforementioned object, it is appropriate to apply \Cref{cor:polynomial} to the matrix
\[
	F^N=\bigl(\sqrt{N} \abs[\big]{U^N_i(\theta)}\bigr)_{(i,\theta)\in[2]\times\Theta_N},
\]
setting \(q=2\) and \(n\equiv 2\) (see the final claim in the Remark following Theorem \ref{thm:main}).

\begin{proposition}
\label{prop:stat}
	Suppose that for all \(N\in\N\), \(X^N=(X^N_{i,\theta})_{(i,\theta)\in[2]\times\Theta_N}\) is a \(2\times m\) centered Gaussian random matrix whose columns have the same covariance matrix as the respective columns of the random matrix \(F^N\), and denote by $\abs{X}^N$ the matrix whose entries are given by the absolute values of the corresponding entries of $X^N$.
	Then there is an absolute constant \(c>0\) 
	\[
		\sup_{x\in\R}\abs[\big]{\Prob(\min\max \abs{X}^N\le x)-\Prob(\min\max F^N\le x)}
		\le c\,\frac{\log^6 m}{N}.
	\]
\end{proposition}
\begin{proof}
Note that the construction of \(\mathcal{T}_B\) and \(\mathcal{T}_{W^\theta}\) ensures that assumptions [A1] and [A2] of \cite[Section 4.1]{K} are met. Moreover, as in \cite[Lemma B.7]{K}, we have that
\[
	\max_{i,\theta}\Ex( (\sqrt{N}U^N_i(\theta))^4)-3\Ex((\sqrt{N}U^N_i(\theta))^2)^2\le\frac{c}{N}
\]
for some constant \(c>0\) depending only on \(\rho_1\) and \(\rho_2\). Since \(n\) is fixed and \(q=2\), we recover the claimed inequality from \Cref{cor:polynomial}.
\end{proof}
\Cref{prop:stat} implies that \(\min\max F^N\) is asymptotically close to the min-max of a suitable Gaussian random matrix as long as \(\log \card{\Theta_N}=o(N^{1/6})\).

\medskip

\noindent\textit{Remark.} There is no conceptual difficulty in extending the previous convergence results to the case in which the correlation between $B_i $ and $\tilde B_i$ equals some non zero parameter $\rho_i$, $i=1,2$. In this case, given a fixed nonzero $a\in (-1,1)$, the corresponding modification of the statistics $U_1^N(\theta), \, U_2^N(\theta)$ can in principle be used to solve the following statistical hypothesis testing problem:

\begin{align*}
	H_0&: \rho_1=0 \text{ or } \rho_2=0, \quad\text{(null hypothesis)}\\
	H_1&: \rho_1 =\rho_2 =a.
\end{align*}
We regard this line of investigation as a separate topic, and leave it open for further investigation. 

\printbibliography
\end{document}